\documentclass[11pt,leqno]{article}
\usepackage{graphicx}
\usepackage{amsfonts, amsthm}
\usepackage{amsxtra}
\usepackage{tikz}
\usepackage[latin1]{inputenc}
\usepackage{fontenc}
\usepackage[dvips]{epsfig}
\begin{document}
\newtheorem{theo}{Theorem}[section]
\newtheorem{defi}[theo]{Definition}
\newtheorem{lemm}[theo]{Lemma}
\newtheorem{prop}[theo]{Proposition}
\newtheorem{rem}[theo]{Remark}
\newtheorem{exam}[theo]{Example}
\newtheorem{cor}[theo]{Corollary}
\newcommand{\mat}[4]{
    \begin{pmatrix}
           #1 & #2 \\
           #3 & #4
      \end{pmatrix}
   }
\def\Z{\mathbb{Z}} 
\def\R{\mathcal{R}} 
\def\I{\mathcal{I}} 
\def\C{\mathbb{C}} 
\def\N{\mathbb{N}} 
\def\PP{\mathbb{P}} 
\def\Q{\mathbb{Q}} 
\def\L{\mathcal{L}} 
\def\ol{\overline} 
\def\bs{\backslash} 
\def\part{P} 

\newcommand{\gcD}{\mathrm {\ gcd}} 
\newcommand{\End}{\mathrm {End}} 
\newcommand{\Aut}{\mathrm {Aut}} 
\newcommand{\GL}{\mathrm {GL}} 
\newcommand{\SL}{\mathrm {SL}} 
\newcommand{\PSL}{\mathrm {PSL}} 
\newcommand{\Mat}{\mathrm {Mat}} 
\newcommand\ga[1]{\overline{\Gamma}_0(#1)} 
\newcommand\pro[1]{\mathbb{P}1(\mathbb{Z}_{#1})} 
\newcommand\Zn[1]{\mathbb{Z}_{#1}}
\newcommand\equi[1]{\stackrel{#1}{\equiv}}
\newcommand\pai[2]{[#1:#2]} 
\newcommand\modulo[2]{[#1]_#2} 
\newcommand\sah[1]{\lceil#1\rceil} 
\def\sol{\phi} 
\begin{center}
{\LARGE\bf Renormalization of circle diffeomorphisms
with a break-type  singularity}
\\
\vspace{.25in} { \large{Habibulla Akhadkulov\footnote{School of Quantitative Sciences, University Utara Malaysia, CAS 06010, UUM  Sintok, Kedah Darul Aman, Malaysia.  \quad E-mail: akhadkulov@yahoo.com },
Mohd Salmi Md Noorani\footnote{School of Mathematical Sciences Faculty of Science and Technology University Kebangsaan Malaysia, 43600 UKM Bangi, Selangor Darul Ehsan, Malaysia.  \quad E-mail:  msn@ukm.edu.my (M.S. Noorani), akhatkulov@yahoo.com (S. Akhatkulov).}
   and}
Sokhobiddin Akhatkulov$^{3}$}

\end{center}

\title{Renormalization of circle diffeomorphisms
with a break-type  singularity}

\begin{abstract}
Let $f$ be an orientation-preserving circle diffeomorphism with irrational rotation number and with a break point
$\xi_{0},$ that is, its derivative $f'$ has a jump discontinuity at this point.
 Suppose that  $f'$ satisfies a certain Zygmund condition dependent on a parameter $\gamma>0.$
  We prove that the renormalizations of $f$ are approximated by  M\"{o}bius transformations
in $C^{1}$-norm if $\gamma\in (0,1]$ and
they are approximated
in $C^{2}$-norm if  $\gamma\in (1,+\infty).$  It is also shown, that the coefficients of  M\"{o}bius transformations get asymptotically linearly
dependent.
\end{abstract}
\section{Introduction and results}
One of the most studied classes of dynamical systems are orientation-preserving
 diffeomorphisms  of the circle $\mathbb{S}^{1}=\mathbb{R}/\mathbb{Z}$. Poincar\'{e} (1885) noticed that the orbit structure of orientation-preserving
diffeomorphism   $f$ is determined by some irrational mod 1, called the \emph{ rotation number} of $f$ and denoted $\rho=\rho(f),$
 in the following sense: for any $\xi\in \mathbb{S}^1,$ the mapping $f^{j}(\xi)\rightarrow j\rho \mod 1,$ $j\in \mathbb{Z},$
is orientation-preserving. Denjoy proved, that if $f$ is an orientation-preserving $C^{1}$-diffeomorphism of the circle with
 irrational rotation number $\rho$ and $\log f'$ has bounded variation then, the orbit
  $\{f^{j}(\xi)\}_{j\in \mathbb{Z}}$  is dense and the mapping $f^{j}(\xi)\rightarrow j\rho \mod 1$
  can therefore be extended by continuity to a homeomorphism $h$ of $\mathbb{S}^1,$ which conjugates $f$
  to the rigid  rotation $f_{\rho}:\xi\rightarrow \xi+\rho \mod 1.$

In this context, a natural
question is to ask  under what condition one can  obtain the  smoothness
 of conjugacy $h.$ The first local results, that is the results  requiring the closeness of diffeomorphism to
the rigid  rotation,   were  obtained by  Arnold \cite{Ar1961} and  Moser \cite{Mo1966}.
Next Herman \cite{He1979} obtained a first global result (i.e. not requiring the closeness of diffeomorphism to
the rigid rotation)  asserting regularity of conjugacy of the circle diffeomorphism.
Further, his result was developed by Yoccoz \cite{Yo1984},
Stark \cite{Stark}, Khanin \& Sinai \cite{KS1987, KS1989},
Katznelson \& Ornstein \cite{KO1989.1, KO1989.2} and Khanin \& Teplinsky \cite{KT2009}. They have shown that if $f$ is $C^3$ or $C^{2+\nu}$
and $\rho$ satisfies a certain Diophantine condition then the conjugacy will be at least $C^1$.

 Note, the renormalization approach in \cite{KS1989} and  \cite{Stark}  to Herman's theory is more natural.
 In this approach,  the regularity of conjugacy statement can be obtained using the convergence of renormalizations
 of sufficiently smooth circle diffeomorphisms with the same irrational rotation number.
 In fact, the renormalizations of circle diffeomorphisms approach to a family of linear maps with slope 1. Such
  convergence implies the regularity of conjugacy if the rotation number satisfies the certain Diophantine condition.

A natural generalization of diffeomorphisms of the circle are diffeomorphisms
with break points, i.e., those circle diffeomorphisms which are smooth everywhere with the exception
of finitely many points  at which their derivatives have  jump discontinuities.
Circle diffeomorphisms with breaks were investigated by Herman \cite{He1979}  in the
piecewise-linear (PL) case. General (non PL) circle diffeomorphisms with a break point was studied  by
Khanin \& Vul \cite{KV1991, VK1990} at the beginning of 90's.

It turns out that, the renormalizations of circle diffeomorphisms with  break points
 are rather different from those of smooth  diffeomorphisms.
 Indeed, the renormalizations of such diffeomorphisms   approach  exponentially to a
 two-dimensional space of  M\"{o}bius transformations (see Theorem \ref{VK})
 with very non-trivial dynamics.
 In the sequel we  provide  some basic utilizations of Khanin \& Vul's \cite{VK1990} result
 in  the \emph{rigidity} theory.  Rigidity, in this context, is the phenomenon of smooth conjugacy
between any two maps which a priori are only topologically equivalent.

Since the renormalizations of diffeomorphisms with breaks, approach
to a two-dimensional space of M\"{o}bius  transformations,
 it is convenient to study the action of renormalization operator
 on a space of pairs (for particular size of a break) in $\mathbb{R}^{2}$ which correspond to M\"{o}bius  pairs.
The investigations  of those  M\"{o}bius transformations in  \cite{KhKhem2003}, \cite{KT2013} and  \cite{KhYam}
showed that the renormalization operator possesses strong hyperbolic properties in a certain domain of that space,
which are analogous to those predicted by Lanford \cite{Lan1988} in  case of critical rotations.

The hyperbolicity of renormalization operator  and as well as analysing the M\"{o}bius
transformations give exponential convergence of renormalizations of two circle diffeomorphisms with a break in
$C^2$-topology \cite{KhKhem2003}, \cite{KT2013}. Very recently, the renormalization
conjecture, that is any two smooth circle diffeomorphisms with a break, with the same
 irrational rotation number and same size of the break, belong to the same universality class,
 (i.e., their renormalizations approach each other)
 has been proven by Khanin \& Koci\'{c}  \cite{KSasa}.
 This conjecture for $C^{\infty}$-critical circle maps with bounded type of rotation numbers
 was proven by de Faria \& de Melo \cite{MeloFarie I, MeloFarie II}
 and extended in the analytic setting, to all rotation numbers by Yampolsky \cite{Yampol}.
Another aspect of   Khanin \& Vul's result is,
their result   plays an important
role to show uniformly regularity (see definition \cite{KSE}, \cite{KT2013})
of renormalizations.

Hence, the validity of above  statements ensure $C^{1}$-rigidity of circle maps with singularities.
Note that, the rigidity results for circle diffeomorphisms with breaks have been obtained by Khanin \& Khmelov \cite{KhKhem2003}
for a countable set of irrational rotations numbers and for a zero measure
set of irrational rotations numbers by Khanin \& Teplinsky \cite{KT2013}.
The rigidity problem was recently completely solved by Khanin \emph{et al.}  \cite{KSE} for almost all rotation numbers.
In the case of critical circle maps de Faria \& de Melo \cite{MeloFarie I, MeloFarie II} proved that
for a full measure  set of irrational rotation numbers,  $C^{3}$-smooth critical maps
with odd integer order of the critical point, the conjugacy is  $C^{1+\alpha}$ for some $\alpha>0.$
And any two real-analytic critical circle maps with the same irrational rotation number
of bounded type are $C^{1+\nu}$ conjugate for some $\nu>0.$
Recently,  Khanin \& Teplinsky \cite{KT2007} proved that for real-analytic critical circle maps
and for any irrational rotation numbers the rigidity holds,
this  rigidity is called robust rigidity.

The purpose of this work is to extend Khanin \& Vul's result \cite{VK1990}.
For this we consider a class of circle  diffeomorphisms
 with  break points satisfying a certain Zygmund condition depending
 on a parameter $\gamma>0.$  The class of such diffeomorphisms  is wider than $C^{2+\nu}.$
 In our first main Theorem \ref{main1} we show that, if $\gamma\in (0,1]$
then the renormalizations approach to  M\"{o}bius transformations
with  the rate of  $\mathcal{O}(n^{-\gamma})$ in $C^{1}$-topology.
In the case of $\gamma\in (1, +\infty)$ this class is a subset of $C^{2},$
therefore we  investigate the  renormalizations in $C^{2}$-topology.
In the second main Theorem \ref{coefbog}  we show that, if
$\gamma\in (1, +\infty)$ then the renormalizations approach to  M\"{o}bius transformations
with  the rate of  $\mathcal{O}(n^{-\gamma})$ in $C^{1}$-topology.
Moreover, the second derivative of the renormalizations approach to the second derivative 
of those M\"{o}bius transformations with  the rate of  $\mathcal{O}(n^{-(\gamma-1)})$ in $C^{0}$-topology.
 It is also shown that, the coefficients of  M\"{o}bius transformations
get asymptotically linearly dependent.

\subsection{Renormalizations of circle diffeomorphisms  with a break point}
 Let $f:\mathbb{S}^{1}\rightarrow \mathbb{S}^{1}$ be a circle diffeomorphism  with a single break point
 $\xi_{0}$ i.e.,  $f$ satisfies the following
conditions:
 \begin{itemize}
   \item [(i)] $f\in C^{1}([\xi_{0}, \xi_{0}+1]);$
   \item [(ii)] $\inf_{\xi\neq \xi_{0}} f'(\xi)>0;$
   \item [(iii)] $f$ has  one-sided derivatives $f'(\xi_{0}\pm 0)>0$ and
   $$
   c:=\sqrt{\frac{f'(\xi_{0}-0)}{f'(\xi_{0}+0)}}\neq 1.
   $$
 \end{itemize}
 The number $c$ is called a \emph{size of break} of $f'$ at  $\xi_{0}.$
 Below we briefly  recall  the definition of renormalization  and
formulate some obtained results.  Let the rotation number  $\rho$ be an irrational.
 We use the following \emph{continued fraction} expansion of rotation number
$$
\rho=1/(k_{1}+1/(k_{2}+...)):=[k_{1}, k_{2},...,k_{n},...).
$$
The sequence of positive integers $(k_{n})$ with $n\geq 1$  called
 \emph{partial quotients} and it is infinite if and only if $\rho$ is irrational.
 Every irrational $\rho$ defines uniquely  the
sequence of partial quotients. Conversely, every infinite sequence of partial quotients defines
uniquely an irrational number $\rho$ as the limit of the sequence of rational convergents
$p_{n}/q_{n}=[k_{1}, k_{2},...,k_{n}].$ The coprime numbers $p_{n}$ and $q_{n}$
satisfy the recurrence relations $p_{n}=k_{n}p_{n-1}+p_{n-2}$ and
$q_{n}=k_{n}q_{n-1}+q_{n-2}$ for $n\geq 1,$ where, for convenience we set
$p_0=0,$ $q_0=1$ and $p_{-1}=1,$ $q_{-1}=0.$
Taking the break point  $\xi_{0}\in \mathbb{S}^{1},$  we define the
 $n$th \emph{fundamental segment} $I^{n}_{0}:=I^{n}_{0}(\xi_{0})$ as the circle arc
  $[\xi_{0}, f^{q_{n}}(\xi_{0})]$ if $n$ is even and $[f^{q_{n}}(\xi_{0}), \xi_{0}]$ if $n$ is odd.
The union of two consequent fundamental segments $I^{n-1}_{0},$  $I^{n}_{0}$
is called the $n$th \emph{renormalization neighborhood} of the point $\xi_{0}$ and we denote
it by $\mathcal{V}_{n}.$
Certain number of images of fundamental segments $I^{n-1}_{0}$ and $I^{n}_{0},$ under  iterates of
 $f,$ cover whole circle without overlapping beyond the endpoints and form
the $n$th \emph{dynamical partition} of the circle
$$
\mathbb{P}_{n}:=\mathbb{P}_{n}(\xi_{0}, f)=\left\lbrace I_{j}^{n}:=f^{j}(I^{n}_{0}), \ 0\leq j<q_{n-1}\right\rbrace
\cup \left\lbrace I_{i}^{n-1}:=f^{i}(I^{n-1}_{0}), \ 0\leq i<q_{n}\right\rbrace.
$$
On $\mathcal{V}_{n}$   we
 define  the Poincar\'{e} map
$\pi_{n}=(f^{q_n}, f^{q_{n-1}}):\mathcal{V}_{n}\rightarrow \mathcal{V}_{n}$ as follows.
\[\pi_{n}(\xi)=\left\{\begin{array}{ll}f^{q_n}(\xi), & \mbox{if \,\,
$\xi\in I^{n-1}_{0}$},\\ f^{q_{n-1}}(\xi), & \mbox{if \,\,
$\xi\in I^{n}_{0}$}.\end{array}\right.\]
The main idea of renormalization  method is to study
the behaviour of the  Poincar\'{e} map
$\pi_{n}$ as $n\rightarrow \infty.$
For this, rescaling the coordinates are usually used.
Let $\mathcal{A}_{n}:\mathbb{R}\rightarrow \mathbb{S}^{1}$ be an affine
covering map such that $\mathcal{A}_{n}([-1, 0])=I_{0}^{n},$ with
$\mathcal{A}_{n}(0)=\xi_{0}$ and $\mathcal{A}_{n}(-1)=f^{q_{n}}(\xi_{0}).$
Define $a_{n}\in \mathbb{R}$ to be a positive number
such that $\mathcal{A}_{n-1}(a_{n})=f^{q_{n}}(\xi_{0}).$
 Obviously, $\mathcal{A}_{n-1}:[0, a_{n}]\rightarrow I^{n}_{0}$
and $\mathcal{A}_{n-1}:[-1, 0]\rightarrow I^{n-1}_{0}.$
Consider a mapping $\mathcal{S}_{n}:[-1, a_{n}]\rightarrow [-1, a_{n}]$
 defined by $\mathcal{S}_{n}:=(f_{n}, g_{n})=\mathcal{A}^{-1}_{n-1}\circ \pi_{n}\circ \mathcal{A}_{n-1},$
where $\mathcal{A}^{-1}_{n-1}$ is the inverse branch that
maps $\mathcal{V}_{n}$  onto $[-1, a_{n}].$
The pair of functions $(f_{n}, g_{n})$  is called the $n$th \emph{renormalization} of $f$
with respect to $\xi_{0}.$
Now we provide the exact statement of Khanin  \& Vul's  \cite{VK1990} result
 asserting the closeness of $(f_{n}, g_{n})$  to the M\"{o}bius transformations $F_{n}:=F_{a_{n}, b_{n},m_{n}}$ and
$G_{n}:=G_{a_{n}, b_{n},m_{n},c_{n}}$  where
\begin{equation}\label{eq6h4}
F_{n}(z)=\frac{a_{n}+(a_{n}+b_{n}m_{n})z}{1+(1-m_{n})z}, \,\,\,\,\,\,\,
G_{n}(z)=\frac{-a_{n}c_{n}+(c_{n}-b_{n}m_{n})z}{a_{n}c_{n}+(m_{n}-c_{n})z}
\end{equation}
and
$$
c_{n}=c^{(-1)^n}, \,\,\,\ b_{n}=\frac{|I_{0}^{n-1}|-|I_{q_{n-1}}^{n}|}{|I_{0}^{n-1}|},
\,\,\,\, m_{n}=\exp\Big((-1)^n\sum_{i=0}^{q_{n}-1}\int_{I_{i}^{n-1}}\frac{f''(x)}{2f'(x)}dx\Big).
$$
\begin{theo}\cite{VK1990}\label{VK}
Let $f$ be a $C^{2+\nu}(\mathbb{S}^{1}\setminus \{\xi_{0}\}),$ $\nu>0$  diffeomorphism with a break point $\xi_0$
and with irrational rotation number.  Then there exist constants $C>0$ and $0<\lambda=\lambda(f) <1$ such that
$$
\|f_{n}-F_{n}\|_{C^{2}([-1,0])}\leq C \lambda^{n},\,\,\,\,\,\, \,\,\,\,\,\,
\|g_{n}-G_{n}\|_{C^{1}([0,a_{n}])}\leq C \lambda^{n}
$$
and
$$
\|g''_{n}-G''_{n}\|_{C^{0}([0,a_{n}])}\leq \frac{C \lambda^{n}}{a_{n}}.
$$
Moreover,
$$
|a_{n}+b_{n}m_{n}-c_{n}|\leq Ca_{n}\lambda^{n}.
$$
\end{theo}
\textbf{Remark.} In the case of rational rotation numbers, the renormalizations of circle diffeomorphisms with breaks
 was investigated  by Khanin \& Vul also.
They analyzed  periodic trajectories of
 renormalization operator on one parameter family.
Moreover, they showed that the Lebesgue measure
of the set of parameters which correspond
to the rational rotation numbers is full.
Later on this result was generalized by  Khmelev
\cite{Khmelev} for circle diffeomorphisms  with several break points.

\textbf{Remark.} Recently, Cunha \& Smania \cite{CS2013} have  studied
Rauzy-Veech renormalizations of $C^{2+\nu}$-circle diffeomorphisms with
several break points.
The main idea of this work is to consider the piecewise-smooth circle homeomorphisms as
generalized interval exchange transformations. They have proved that Rauzy-Veech renormalizations of
$C^{2+\nu}$-generalized interval exchange maps satisfying
a certain  combinatorial conditions are approximated by M\"{o}bius
transformations in $C^2$-topology.
\subsection{Main results}
  To state  our main results,  we define a new class of circle diffeomorphisms
  with one break point.  Consider the function  $\mathcal{Z}_{\gamma}:[0,1)\rightarrow (0, +\infty),$
given
  $$
  \mathcal{Z}_{\gamma}(x)=\frac{1}{(\log\frac{1}{x})^{\gamma}}, \,\,\,\,\,\,\,\,x\in (0,1)
  $$
and $\mathcal{Z}_{\gamma}(0)=0,$ where $\gamma>0.$
 Let $f$ be a circle diffeomerphism
with the break point $\xi_{0}.$ Without loss of generality we may assume
$\xi_{0}=0.$ Denote by $\Delta^{2}f'(\xi, \tau)$
the \emph{second symmetric difference} of $f',$ that is
$$
\Delta^{2}f'(\xi, \tau)=f'(\xi+\tau)+f'(\xi-\tau)-2f'(\xi)
$$
where  $\xi \in \mathbb{S}^{1}$ and  $\tau\in [0,\frac{1}{2}].$
Suppose that there exists a constant $C>0$ such that
\begin{equation}\label{ok1}
 \|\Delta^{2}f'(\cdot, \tau)\|_{L^{\infty}(\mathbb{S}^{1})}\leq C\tau\mathcal{Z}_{\gamma}(\tau).
\end{equation}
Note that the class of real  functions satisfying (\ref{ok1})
with $\mathcal{Z}_{\gamma}(\tau)$ replaced by  1,  is called the Zygmund class and denoted by $\Lambda_{\ast}$
(see \cite{Zyg}, p. 43). The  class  $\Lambda_{\ast}$
plays a key role to investigate  trigonometric series and this class
 was applied  to the theory of circle homeomorphisms
for the first time by Jun Hu \& Sullivan (see \cite{JSull1997}, \cite{Sull1992}).
 They extended the classical Denjoy's theorem to this class.
Note that  if  $f'$ satisfies (\ref{ok1})
 then it is not necessarily of bounded variation  and vice versa.
Indeed, there are examples in \cite{MeloBook} and \cite{Zyg}
for this statement.
  In this work we study the class of circle diffeomerphisms $f$  with the break point $\xi_{0},$
 whose  derivatives $f'$ have bounded variation and
satisfy  the  inequality (\ref{ok1}). And we denote this class by
$\mathrm{D}^{1+\mathcal{Z}_{\gamma}}(\mathbb{S}^{1}\setminus \{\xi_{0}\}).$
 Let $f\in \mathrm{D}^{1+\mathcal{Z}_{\gamma}}(\mathbb{S}^{1}\setminus \{\xi_{0}\})$
 and its rotation number is irrational.
 We define two quantities  $\widetilde{m}_{n}$  and $\widehat{m}_{n}$
  as
$$
\widetilde{m}_{n}=\exp\Big(\sum_{i=0}^{q_{n}-1}\frac{f'(\xi_{i})-f'(\xi_{i+q_{n-1}})}{2f'(\xi_{i})}\Big), \,\,\,
\widehat{m}_{n}=\exp\Big(\sum_{j=0}^{q_{n-1}-1}\frac{f'(\xi_{j+q_n})-f'(\xi_{j})}{2f'(\xi_{j+q_n})}\Big)
$$
where $\xi_{i},$ $\xi_{i+q_{n-1}}$ and $\xi_{j},$  $\xi_{j+q_{n}}$  are endpoints of the intervals $I_{i}^{n-1},$
$I_{j}^{n}$ respectively.
Since the systems of intervals $\{I_{i}^{n-1}, \,\,\, 0\leq i <q_{n}\},$
$\{I_{j}^{n}, \,\,\,  0\leq j <q_{n-1}\}$
do not intersect  and $f'$ has bounded variation,  $\widetilde{m}_{n},$ $\widehat{m}_{n}$
are bounded for any $n\geq 1.$
Below we will show that $\widetilde{m}_{n}$ and $\widehat{m}_{n}$ are exponentially
close to $m_{n}$ and $c_{n}m^{-1}_{n}$ respectively, for $\gamma>1.$ Using
$\widetilde{m}_{n},$ $\widehat{m}_{n}$ we define  M\"{o}bius transformations
 similarly as in (\ref{eq6h4}) as follow
  \begin{equation}\label{eq6h5}
  \widetilde{F}_{n}(z)=\frac{a_{n}+(a_{n}+b_{n}\widetilde{m}_{n})z}{1+(1-\widetilde{m}_{n})z},\,\,\,\,\,\,\,\,\,\
  \widehat{G}_{n}(z)=\frac{-a_{n}\widehat{m}_{n}+(\widehat{m}_{n}-b_{n})z}{a_{n}\widehat{m}_{n}+(1-\widehat{m}_{n})z}.
  \end{equation}
  Our first main result is the following.
 \begin{theo}\label{main1}
Let $f \in \mathrm{D}^{1+\mathcal{Z}_{\gamma}}(\mathbb{S}^{1}\setminus \{\xi_{0}\})$ and $\gamma\in (0, 1].$
Suppose the  rotation number of $f$ is irrational.   Then there exists  a constant $C=C(f)>0$ and a
 natural number $N_{0}=N_{0}(f)$  such that   the following inequalities
\begin{equation}\label{Ren1}
\|f_{n}-\widetilde{F}_{n}\|_{C^{1}([-1,0])}\leq \frac{C}{n^{\gamma}},\,\,\,\,\,\, \,\,\,\,\,\,
\|g_{n}-\widehat{G}_{n}\|_{C^{1}([0,a_{n}])}\leq  \frac{C}{n^{\gamma}}
\end{equation}
hold for all $n\geq N_{0}.$
\end{theo}
Note that  the class $\mathrm{D}^{1+\mathcal{Z}_{\gamma}}(\mathbb{S}^{1}\setminus \{\xi_{0}\})$
will be "better" when $\gamma$  increases.
This gives  more opportunities to
 better  understand the behavior of $\mathcal{S}_{n}.$
Now we consider the case $\gamma>1.$
In this case, because of Theorem \ref{WZ}
stated in Section 2, $f'$ is differentiable on $\mathbb{S}^{1}\setminus \{\xi_{0}\},$ hence
$\mathcal{S}'_{n}$ is differentiable on $[-1, a_{n}]\setminus \{0\}.$
 Therefore  we can investigate the  behavior
of $\mathcal{S}''_{n}.$  Our second  result is the following.
\begin{theo}\label{coefbog}
Let $f \in \mathrm{D}^{1+\mathcal{Z}_{\gamma}}(\mathbb{S}^{1}\setminus \{\xi_{0}\})$ and $\gamma>1.$
Suppose the  rotation number of $f$ is irrational.
  Then there exists  a constant $C=C(f)>0$ and a
 natural number $N_{0}=N_{0}(f)$  such that for all $n\geq N_{0}$  the following inequalities hold
\begin{equation}\label{Ren100}
\|f_{n}-F_{n}\|_{C^{1}([-1,0])}\leq \frac{C}{n^{\gamma}},\,\,\,\,\,\, \,\,\,\,\,\,
\|g_{n}-G_{n}\|_{C^{1}([0,a_{n}])}\leq  \frac{C}{n^{\gamma}}.
\end{equation}
\begin{equation}\label{Ren2}
\|f''_{n}-F''_{n}\|_{C^{0}([-1,0])}\leq \frac{C}{n^{\gamma-1}},\,\,\,\,\,\, \,\,\,\,\,\,\,\,
\|g''_{n}-G''_{n}\|_{C^{0}([0,a_{n}])}\leq  \frac{C}{a_{n}n^{\gamma-1}}.
\end{equation}
Moreover,
\begin{equation}\label{ren3}
|a_{n}+b_{n}m_{n}-c_{n}|\leq \frac{Ca_{n}}{n^{\gamma}}
\end{equation}
where $F_{n}$ and $G_{n}$ are defined in (\ref{eq6h4}).
\end{theo}
\textbf{Remark.}
\begin{itemize}
  \item  It is obvious  that the classes
 in Theorems \ref{main1} and \ref{coefbog} are wider than the class of Theorem \ref{VK}
 but the estimations are weaker.
\item  The Zygmund conditions is quite natural in the context of Cross-Ration Distortion (CRD).
 The relations between Zygmund conditions and CRD estimates have been studied  in the book \cite{MeloBook} for $\gamma=1.$
  Since Ratio Distortion (RD) is a partial case of CRD,  the approaches in the above book  work very well
  to estimate RD for the considered Zygmund class.
  On the other hand the renormalizations can be represented by RD, therefore we investigate
  the renormalizations by RD.
  \end{itemize}
The structure of  paper is as follows. In Section 2, we provide brief facts
about Zygmund functions and following Khanin \& Vul  \cite{VK1990}
we define a relative coordinate of an interval.
 Then we obtain some estimates for the distortion of interval.
Moreover, we provide relations between  distortion and  relative coordinates of  intervals.
In Section 3, we get estimates for the ratio of $f^{q_{n}}$-distortion of intervals
i.e., distortion of intervals with respect to $f^{q_{n}}$ for the different $\gamma$'s.
 In Section 4, we compare the relative coordinates with M\"{o}bius transformations.
Finally, in Section 5 we prove our main theorems.

\section{Ratio distortions and Zygmund condition}
\subsection{Notes on Zygmund functions}
In this subsection we provide brief facts about functions satisfying inequality (\ref{ok1}).
These facts will be used in the proof of main results.
 Let $I=[a,b]$ be an interval with  the length less than 1.
 Consider a continuous function
$\mathcal{K}:I\rightarrow\mathbb{R}.$ Suppose $\mathcal{K}$ satisfies the inequality (\ref{ok1})
on $I$ i.e.,
\begin{equation}\label{ok1111}
 \|\Delta^{2}\mathcal{K}(\cdot, \tau)\|_{L^{\infty}([a,b])}\leq C\tau\mathcal{Z}_{\gamma}(\tau),
\end{equation}
where $\tau\in [0, |I|/2].$
It turns out that the functions satisfying relation (\ref{ok1111})
have "a considerable degree of continuity".
\begin{theo}\label{modofcon}
Let $\mathcal{K}:I\rightarrow\mathbb{R}$ be  continuous
and satisfies the inequality (\ref{ok1111}) on $I$.  If $\gamma\in (0,1)$ then
$$
\omega(\delta,\mathcal{K} )=\mathcal{O}\Big(\delta(\log\frac{1}{\delta})^{1-\gamma}\Big).
$$
If $\gamma=1$ then
$$
\omega(\delta,\mathcal{K} )=\mathcal{O}\Big(\delta(\log\log\frac{1}{\delta})\Big)
$$
where $\omega(\cdot, \mathcal{K})$ is the modulus of continuity of $\mathcal{K}.$
\end{theo}
\begin{proof}
The proof of this theorem follows closely that of  \cite{Zyg} (p. 44) and we leave it to the reader.
\end{proof}
The following theorem was proved by
Weiss \& Zygmund in \cite{WZyg1959}. This theorem will be used in the proof of next theorem.
\begin{theo}\label{WZorginal}
Let $\mathcal{K}:\mathbb{R}\rightarrow\mathbb{R}$ be $2\pi$-periodic  and satisfies (\ref{ok1111}) for some  $\gamma \in (\frac{1}{2}, 1].$
Then $\mathcal{K}$  is absolute continuous and $\mathcal{K}\in L_{p}[0,2\pi]$ for every $p>1.$
\end{theo}
More direct and general proof of this theorem can be found in \cite{JN1961}.
In this theorem the assumption $\gamma \in (\frac{1}{2}, 1]$ is crucial.
The theorem is false for   $\gamma \in (0, \frac{1}{2}].$
Indeed, there are functions which satisfy  (\ref{ok1111}) for some $\gamma \in (0, \frac{1}{2}]$
but almost nowhere differentiable (see \emph{e.g.} \cite{Zyg}).
Next we provide a theorem on differentiability of  $\mathcal{K}$ in the case of $\gamma>1.$
To state this theorem we need the following function $\mathcal{P_{\gamma}}:(0,1)\rightarrow \mathbb{R}.$
\begin{equation}\label{eq6303333avgust}
\mathcal{P_{\gamma}}(x)=\sum_{n=1}^{\infty}\mathcal{Z}_{\gamma}(x 2^{-n}) \,\,\,\,\text{where}\,\,\, x\in(0,1)
\,\,\,\,\,\text{and}\,\,\,\gamma>1.
\end{equation}
It is clear that $\mathcal{P_{\gamma}}$ is  continuous  and $\underset{x\rightarrow 0}{\lim}\mathcal{P_{\gamma}}(x)=0.$
\begin{theo}\label{WZ}
Let $\mathcal{K}:I\rightarrow\mathbb{R}$ be  continuous and satisfies (\ref{ok1111}) for some  $\gamma>1.$
Then $\mathcal{K}\in C^{1}(I)$ and for any $\xi,$ $\eta\in I$
there exists a constant $C>0$ such that
$$
|\mathcal{K'}(\xi)-\mathcal{K'}(\eta)|\leq C\cdot\mathcal{P_{\gamma}}(|\xi-\eta|).
$$
\end{theo}%
\begin{proof}
We give only the sketch of proof, the details will be left to the reader.
According to  Theorem \ref{WZorginal}   the function $\mathcal{K}$ is
at least absolute continuous on $I$ in the case of $\gamma>1.$
Hence $\mathcal{K'}$ exists almost everywhere and $\mathcal{K}$ is an
indefinite integral of $\mathcal{K'}.$ To prove the theorem
we take any points $\xi,\eta \in I$ that are Lebesgue points of $\mathcal{K'}$
  and using the same manner as in \cite{Zyg} (p. 44), we obtain a uniform estimate for
$|\mathcal{K'}(\xi)-\mathcal{K'}(\eta)|$.
 Hence we show  $\mathcal{K'}$
is uniformly continuous  and satisfies
$$
|\mathcal{K'}(\xi)-\mathcal{K'}(\eta)|\leq C\cdot\mathcal{P_{\gamma}}(|\xi-\eta|)
$$
on its set of Lebesgue points.
Thus, it can be continuously extended to whole interval $I.$
\end{proof}

\subsection{The distortion of interval and relative coordinate}
In this subsection we introduce the distortion
of interval $I=[a, b]$ with respect to continuous and monotone function $\mathcal{K}: I\rightarrow\mathbb{R}.$
We obtain some estimates  for the distortion of interval.
After that  following Vul \& Khanin \cite{VK1990}, we define the relative coordinate of interval
 and  provide relations between distortion and relative coordinate.
These estimates  and relations 
will be used in the proofs of  main theorems.
The \emph{ distortion} of the  interval $I$ with respect to
$\mathcal{K}$ is
$$
\mathcal{R}(I; \mathcal{K})=\frac{|\mathcal{K}(I)|}{|I|}.
$$
The distortion is multiplicative with respect to composition.
 Henceforth,  take any $x\in [a,b]$ and consider  the
 distortions
\begin{equation}\label{ok5}
\mathcal{R}_{a}(x):=\mathcal{R}([a,x]; \mathcal{K})\,\,\,\,\,\text{and}\,\,\,\,\, \mathcal{R}_{b}(x):=\mathcal{R}([x,b]; \mathcal{K}).
\end{equation}
Below we study the distortions $\mathcal{R}_{a}(x)$ and  $\mathcal{R}_{b}(x)$ as the functions of $x\in [a,b].$
 Consider the following function
$\Omega:(0,1)\times (0, +\infty)\rightarrow \mathbb{R},$
$$
\Omega(\delta,\gamma)=\left\{
                \begin{array}{ll}
\delta(\log\frac{1}{\delta})^{1-\gamma} \,\,\,\,\, \text{if} \,\,\,\, (\delta,\gamma)\in (0,1)\times (0,1);\\
 \\
\delta(\log\log\frac{1}{\delta})\,\,\,\,\, \text{if} \,\,\,\, (\delta,\gamma)\in (0,1)\times \{1\};\\
\\
 \delta\,\,\,\,\,\, \,\,\,\,\,\,\,\,\,\,\,\,\,\,\,\,\,\,\,\,\,\,\,\,\,\text{if}\,\,\,\, (\delta,\gamma)\in (0,1)\times (1, +\infty).
        \end{array}
              \right.
$$
In fact the function $\Omega(\delta,\gamma)$ is the modulus of continuity of
the functions satisfying relation (\ref{ok1111}) for the different cases of $\gamma.$

Denote by $\mathrm{D}^{1+\mathcal{Z}_{\gamma}}(I)$
the class of diffeomorphisms $\mathcal{K}$ whose  derivatives  $\mathcal{K}'$
satisfy the inequality (\ref{ok1111}) on  $I.$
In the sequel we prove several lemmas which will be used in the proofs of main theorems.
\begin{lemm}\label{raiovahosila}
Let $\mathcal{K}\in \mathrm{D}^{1+\mathcal{Z}_{\gamma}}(I)$ and $\gamma\in (0,+\infty).$  Then we have
$$
\frac{\mathcal{R}_{a}(x)}{\mathcal{R}_{b}(x)}-1=\frac{\mathcal{K}'(a)-\mathcal{K}'(b)}{2\mathcal{K}'(b)}+
\mathcal{O}\Big(|I|\cdot\mathcal{Z}_{\gamma}(|I|)+|\mathcal{K}'(a)-\mathcal{K}'(b)|\cdot\Omega(|I|,\gamma)\Big).
$$
\end{lemm}
 \begin{proof}
Since $\mathcal{K}'$ satisfies (\ref{ok1111}), similarly as in \cite{MeloBook} (p. 293)  we have
$$
\mathcal{R}_{a}(x)=\frac{\mathcal{K}'(x)+\mathcal{K}'(a)}{2}+\mathcal{O}\Big(|I|\cdot\mathcal{Z}_{\gamma}(|I|)\Big),
\,\,\,
\mathcal{R}_{b}(x)=\frac{\mathcal{K}'(b)+\mathcal{K}'(x)}{2}+\mathcal{O}\Big(|I|\cdot\mathcal{Z}_{\gamma}(|I|)\Big).
$$
The last two relations imply
$$
\frac{\mathcal{R}_{a}(x)}{\mathcal{R}_{b}(x)}-1=\frac{\mathcal{K}'(a)-\mathcal{K}'(b)}{\mathcal{K}'(b)+\mathcal{K}'(x)}+
\mathcal{O}\Big(|I|\cdot\mathcal{Z}_{\gamma}(|I|)\Big).
$$
It is obvious
$$
\frac{\mathcal{K}'(a)-\mathcal{K}'(b)}{\mathcal{K}'(b)+\mathcal{K}'(x)}=
\frac{\mathcal{K}'(a)-\mathcal{K}'(b)}{2\mathcal{K}'(b)}+
|\mathcal{K}'(a)-\mathcal{K}'(b)|\mathcal{O}\Big(|\mathcal{K}'(x)-\mathcal{K}'(b)|\Big).
$$
Hence, the claim of lemma follows  from   Theorems \ref{modofcon}, \ref{WZ} and above relations.
\end{proof}
Note that  according to Theorem \ref{WZ} the function  $\mathcal{K}'$ is
 differentiable   in the case of $\gamma>1.$  Hence
 we have the following. %
\begin{cor}\label{ratiovahosilarem}
Let $\mathcal{K}\in \mathrm{D}^{1+\mathcal{Z}_{\gamma}}(I)$ for some $\gamma\in (1,+\infty).$  Then we have
$$
\frac{\mathcal{R}_{a}(x)}{\mathcal{R}_{b}(x)}-1=-\int_{a}^{b}\frac{\mathcal{K}''(y)}{2\mathcal{K}'(y)}dy+
\mathcal{O}\Big(|I|\cdot\mathcal{Z}_{\gamma}(|I|)+
|\mathcal{K}'(a)-\mathcal{K}'(b)|\cdot\Omega(|I|,\gamma)\Big).
$$
\end{cor}
Now we define the \emph{relative coordinate} $z:I\rightarrow [0,1]$ as follows
 $$
 z=\frac{b-x}{b-a}.
 $$
Next we prove the following.
\begin{lemm}\label{qavariqlik}
Let $\mathcal{K}\in \mathrm{D}^{1+\mathcal{Z}_{\gamma}}(I)$ and $\gamma\in (0,+\infty).$  Then we have
\begin{equation}\label{21a}
(x-a)(b-x)\Big(\frac{\mathcal{R}'_{b}(x)-\mathcal{R}'_{a}(x)}{b-a}\Big)=
\frac{1}{2}\Big(z\mathcal{K}'(a)+(1-z)\mathcal{K}'(b)-\mathcal{K}'(x)\Big)+\mathcal{O}\Big(|I|\cdot\mathcal{Z}_{\gamma}(|I|)\Big).
\end{equation}
\end{lemm}
 \begin{proof}
 By differentiating $\mathcal{R}_{a}$ and  $\mathcal{R}_{b}$ we obtain
 $$
 \mathcal{R}'_{a}(x)=\frac{\mathcal{K}'(x)-\mathcal{R}_{a}(x)}{x-a}
 \,\,\,\,\,\,\,\,\,\,\text{and}\,\,\,\,\,\,\,\,\,\,
 \mathcal{R}'_{b}(x)=\frac{\mathcal{R}_{b}(x)-\mathcal{K}'(x)}{b-x}.
 $$
Applying the same arguments as in the proof of Lemma \ref{raiovahosila},
we can show that
$$
 \mathcal{R}'_{a}(x)=
 \frac{1}{2}\cdot\frac{\mathcal{K}'(x)-\mathcal{K}'(a)}{x-a}+\frac{1}{1-z}\mathcal{O}\Big(\mathcal{Z}_{\gamma}(|I|)\Big),\,\,\,
 \mathcal{R}'_{b}(x)=
 \frac{1}{2}\cdot\frac{\mathcal{K}'(b)-\mathcal{K}'(x)}{b-x}+\frac{1}{z}\mathcal{O}\Big(\mathcal{Z}_{\gamma}(|I|)\Big).
 $$
 Hence
$$
(x-a)(b-x)\Big(\frac{\mathcal{R}'_{b}(x)-\mathcal{R}'_{a}(x)}{b-a}\Big)=
\frac{1}{2}\Big(z\mathcal{K}'(a)+(1-z)\mathcal{K}'(b)-\mathcal{K}'(x)\Big)+\mathcal{O}\Big(|I|\cdot\mathcal{Z}_{\gamma}(|I|)\Big).
$$
 Lemma \ref{qavariqlik} is proved.
\end{proof}
Next we estimate the expression in the right hand side of equation (\ref{21a}).
For this we define the following function $T_{\gamma}:[0, 1/2]\times (0, 1)\rightarrow \mathbb{R}$
as
$$
T_{\gamma}(s, t)=s\int_{s}^{1}\frac{\mathcal{Z}_{\gamma}(xt)dx}{x}+
\int_{0}^{s}\mathcal{Z}_{\gamma}(xt)dx \,\,\,\,\text{if}\,\,\,s\in (0,1/2]
$$
and $T_{\gamma}(0, t)=0,$  for any $t\in (0,1).$
\begin{lemm}\label{qavariqlik1}
Let $\mathcal{K}\in \mathrm{D}^{1+\mathcal{Z}_{\gamma}}(I)$ and  $\gamma\in (0,+\infty).$
 Then there exists a constant $C>0$ such that
$$
|z\mathcal{K}'(a)+(1-z)\mathcal{K}'(b)-\mathcal{K}'(x)|\leq
C\left\{
                \begin{array}{ll}
|I|\cdot T_{\gamma}(z, |I|) \,\,\,\,\,\,\,\,\ \text{if} \,\,\,\, z\in [0, \frac{1}{2}];\\
\\
|I|\cdot T_{\gamma}(1-z, |I|) \,\,\,\,\text{if}\,\,\,\, z\in (\frac{1}{2}, 1].
        \end{array}
              \right.
$$
\end{lemm}
\begin{proof}
Let us consider  the function $\kappa(z):=\mathcal{K}'(b+z(a-b)),$ $z\in [0, 1].$
It is clear that, to prove the lemma we have to estimate
$|(1-z)\kappa(0)+z\kappa(1)-\kappa(z)|.$
Note that, since $\mathcal{K}'$ satisfies (\ref{ok1111}) we have
\begin{equation}\label{ok16}
  |\frac{1}{2}\kappa(\xi+\tau)+\frac{1}{2}\kappa(\xi-\tau)-\kappa(\xi)|
\leq C\cdot|\tau||I|\mathcal{Z}_{\gamma}(|\tau||I|),
\end{equation}
for all $\xi+\tau,$ $\xi-\tau \in [0, 1].$
Denote by $\mathcal{D}_{\ell},$ $\ell=0,1,2,...$ the dyadic partition of $[0,1].$
Take any $z\in (0, 1),$ fix it and denote by $J_{\ell}=[\mathfrak{a}_{\ell}, \mathfrak{b}_{\ell})$ the dyadic interval of
$\mathcal{D}_{\ell}$ which contains $z.$ Using this interval we define the following quantities
$$
\mathfrak{r}_{\ell}=\frac{\mathfrak{b}_{\ell}-z}{\mathfrak{b}_{\ell}-\mathfrak{a}_{\ell}}\kappa(\mathfrak{a}_{\ell})+\frac{z-\mathfrak{a}_{\ell}}
{\mathfrak{b}_{\ell}-\mathfrak{a}_{\ell}}\kappa(\mathfrak{b}_{\ell})-\kappa(z),\,\,\,\,\,
\mathfrak{t}_{\ell}=\frac{1}{2}\kappa(\mathfrak{a}_{\ell})+\frac{1}{2}\kappa(\mathfrak{b}_{\ell})-\kappa(\frac{\mathfrak{a}_{\ell}+\mathfrak{b}_{\ell}}{2}).
$$
After an easy computation we get
$$
\mathfrak{r}_{\ell+1}-\mathfrak{r}_{\ell}=-\frac{\mathrm{dist}(z, \partial(J_{\ell}))}{|J_{\ell+1}|}\mathfrak{t}_{\ell}.
$$
Using this equality we obtain
$$
|\mathfrak{r_{0}}|\leq\sum_{\ell=0}^{\infty}\frac{\mathrm{dist}(z, \partial(J_{\ell}))}{|J_{\ell+1}|}|\mathfrak{t}_{\ell}|.
$$
There are  two possibilities for $z:$
either  $z\in (0,\frac{1}{2}]$ or $z\in (\frac{1}{2}, 1).$ Consider the first case.
Let   $m$ be the biggest natural number  such that $z\leq 2^{-m}.$  It is easy to verify  that
$$
\mathrm{dist}(z, \partial(J_{\ell}))\leq
\left\{
  \begin{array}{ll}
    z, & \hbox{if} \,\,\,\,\, \ell\leq m;\\
    |J_{\ell}|, & \hbox{if} \,\,\,\,\, \ell>m.
  \end{array}
\right.
$$
Similarly, in the second  case we define the biggest natural number $p$
 such that $1-z \leq 2^{-p}.$ One can verify that
$$
\mathrm{dist}(z, \partial(J_{\ell}))\leq
\left\{
  \begin{array}{ll}
    1-z, & \hbox{if} \,\,\,\,\, \ell\leq p;\\
    |J_{\ell}|, & \hbox{if} \,\,\,\,\, \ell>p.
  \end{array}
\right.
$$
Consequently,
 $$
|\mathfrak{r_{0}}|\leq \left\{
                         \begin{array}{ll}
                           z\underset{\ell=0}{\overset{m}{\sum}}2^{\ell+1}|\mathfrak{t}_{\ell}|+2\underset{\ell=m+1}{\overset{\infty}{\sum}}|\mathfrak{t}_{\ell}|, & \hbox{if}\,\,\,\,\,\,z\in(0,\frac{1}{2}]; \\
\\
                           (1-z)\underset{\ell=0}{\overset{p}{\sum}}2^{\ell+1}|\mathfrak{t}_{\ell}|+2\underset{\ell=p+1}{\overset{\infty}{\sum}}|\mathfrak{t}_{\ell}|, & \hbox{if}\,\,\,\,\,\,z\in(\frac{1}{2}, 1).
                         \end{array}
                     \right.
$$
Using inequality (\ref{ok16}) we get
$$
|\mathfrak{r_{0}}|\leq C \left\{
                         \begin{array}{ll}
                           |I|\Big(z\underset{\ell=0}{\overset{m}{\sum}}\mathcal{Z}_{\gamma}(2^{-\ell-1}|I|)
+\underset{\ell=m+1}{\overset{\infty}{\sum}}2^{-\ell}\mathcal{Z}_{\gamma}(2^{-\ell-1}|I|)\Big), & \hbox{if}\,\,\,\,\,\,z\in(0,\frac{1}{2}]; \\
\\
               |I|\Big((1-z)\underset{\ell=0}{\overset{p}{\sum}}\mathcal{Z}_{\gamma}(2^{-\ell-1}|I|)
+\underset{\ell=p+1}{\overset{\infty}{\sum}}2^{-\ell}\mathcal{Z}_{\gamma}(2^{-\ell-1}|I|)\Big), & \hbox{if}\,\,\,\,\,\,z\in(\frac{1}{2},1).
                         \end{array}
                       \right.
$$
It is obvious that if $z\in(0,\frac{1}{2}]$ then
$$
z\underset{\ell=0}{\overset{m}{\sum}}\mathcal{Z}_{\gamma}(2^{-\ell-1}|I|)
+\underset{\ell=m+1}{\overset{\infty}{\sum}}2^{-\ell}\mathcal{Z}_{\gamma}(2^{-\ell-1}|I|\leq C \cdot T(z, |I|).
$$
 Similarly, if $z\in (\frac{1}{2}, 1)$ then
$$
(1-z)\underset{\ell=0}{\overset{p}{\sum}}\mathcal{Z}_{\gamma}(2^{-\ell-1}|I|)
+\underset{\ell=p+1}{\overset{\infty}{\sum}}2^{-\ell}\mathcal{Z}_{\gamma}(2^{-\ell-1}|I|\leq C\cdot T(1-z, |I|).
$$
 Thus
$$
|\mathfrak{r_{0}}|\leq C \left\{
                         \begin{array}{ll}
                           |I|\cdot T_{\gamma}(z, |I|), & \hbox{if}\,\,\,\,\,\,z\in(0,\frac{1}{2}]; \\
\\
               |I|\cdot T_{\gamma}(1-z, |I|), & \hbox{if}\,\,\,\,\,\,z\in(\frac{1}{2},1).
                         \end{array}
                       \right.
$$
On the other hand
$$
|(1-z)\kappa(0)+z\kappa(1)-\kappa(z)|=|\mathfrak{r_{0}}|.
$$
Hence we have proved Lemma \ref{qavariqlik1} for $z\in (0, 1).$ For  $z\in \{0,1\}$
 the claim of lemma is obvious.
\end{proof}
It is easy to see that the function $T_{\gamma}(z, |I|)$ is  increasing function of $z$ on
$[0,\frac{1}{2}].$ Hence the function $T_{\gamma}(1-z, |I|)$ is decreasing  function of $z$ on
$[\frac{1}{2}, 1].$ Therefore $T_{\gamma}(z, |I|)\leq T_{\gamma}(\frac{1}{2}, |I|)$ for all
$z\in[0,\frac{1}{2}]$ and $T_{\gamma}(1-z, |I|)\leq T(\frac{1}{2}, |I|)$ for all
$z\in[\frac{1}{2},1].$ Moreover, if the length of  interval $I$ is
sufficiently small then it can be easily shown that
$$
T_{\gamma}(\frac{1}{2}, |I|)=\mathcal{O}\Big(\mathcal{Z}_{\gamma}(|I|)\Big).
$$
Thus  Lemmas \ref{qavariqlik} and \ref{qavariqlik1} imply the following.
\begin{cor}\label{corqavariq}
Let $\mathcal{K}\in \mathrm{D}^{1+\mathcal{Z}_{\gamma}}(I)$ and $\gamma\in (0,+\infty).$ If the length of
 interval $I$ is  sufficiently small  then we have
$$
(x-a)(b-x)\Big|\frac{\mathcal{R}'_{b}(x)-\mathcal{R}'_{a}(x)}{b-a}\Big|=\mathcal{O}\Big(|I|\mathcal{Z}_{\gamma}(|I|)\Big).
$$
\end{cor}
Next we consider the subcase $\gamma\in (1, +\infty).$ Due to  Theorem \ref{WZ},
 $\mathcal{K}'$ is differentiable.
Therefore $\mathcal{R}'_{a}$ and $\mathcal{R}'_{b}$ are differentiable and we have the following.
 \begin{lemm}\label{ratiohosilaayirmasi}
Let $\mathcal{K}\in \mathrm{D}^{1+\mathcal{Z}_{\gamma}}(I)$ and $\gamma\in (1, +\infty).$  Then there exists a
constant $C>0$ such that
$$
\Big|(x-a)(b-x)\Big(\mathcal{R}''_{a}(x)-\mathcal{R}''_{b}(x)\Big)\Big|\leq C\cdot|I|\mathcal{P}_{\gamma}(|I|),
$$
$$
\Big|\mathcal{R}'_{b}(x)-\mathcal{R}'_{a}(x)\Big|\leq C\cdot\mathcal{P}_{\gamma}(|I|)
$$
where  function $\mathcal{P}_{\gamma}$ is defined in (\ref{eq6303333avgust}).
\end{lemm}
\begin{proof}
The proof lemma follows from easy computations and Theorem \ref{WZ}.
\end{proof}
\section{Estimates for the ratio of  $f^{q_{n}}$-distortions}
In this section we first  define relative coordinates on the intervals of
dynamical partition $\mathbb{P}_{n}$ and the ratio of $f^{q_{n}}$-distortions i.e.,
distortions of intervals with respect to $f^{q_{n}}.$  Then we
describe the ratio of  $f^{q_{n}}$-distortions by initial relative coordinates and
 provide estimates for this description and its derivatives.
  Note that the relative coordinates on intervals of dynamical partition $\mathbb{P}_{n}$
 was introduced and very well   investigated by Sinai \&  Khanin in the work \cite{KS1989}.
 Here and in the next sections we always assume the rotation number is irrational and
 we use in the following formulations:  $b-a$ where $a, b\in \mathbb{S}^{1}$ and
often the map $f:\mathbb{S}^{1}\rightarrow \mathbb{S}^{1}$ with $f(b)-f(a),$ even if correctly these had to
be replaced by  the lifts $\tilde{b}-\tilde{a}$ where $\tilde{a}, \tilde{b}\in [0,1)$
and $F:\mathbb{R}\rightarrow \mathbb{R}$ with $F(\tilde{b})-F(\tilde{a})$ respectively. Having these in mind
we introduce the relative coordinates $z_{i}:I_{i}^{n-1}\rightarrow [0,1]$ for all $0\leq i \leq q_{n}$
and $\bar{z}_{j}:I_{j}^{n}\rightarrow [0,1]$ for all  $0\leq j \leq q_{n-1},$
 by the formulae respectively:
$$
z_{i}=\frac{\xi_{i}-x}{\xi_{i}-\xi_{i+q_{n-1}}}, \,\,\,\,\, x\in I_{i}^{n-1}
\,\,\,\,\, \text{and} \,\,\,\,\,\
\bar{z}_{j}=\frac{\xi_{j+q_{n}}-y}{\xi_{j+q_{n}}-\xi_{j}}, \,\,\,\,\, y\in I_{j}^{n}.
$$
Here and in the following, we discuss only the case where $n$ is even, the case where $n$ is odd
is obtained by reversing the orientation of $\mathbb{S}^{1}$.  Next we define
\begin{equation}\label{r01}
 \widetilde{\Upsilon}_{n}(x)=\log\frac{\mathcal{R}([\xi_{q_{n-1}}, x];\,\, f^{q_{n}})}{\mathcal{R}([x, \xi_{0}]; \,\, f^{q_{n}})}
 +\log \widetilde{m}_{n},\,\,\,\,\,\,\,\,\,\, x\in I_{0}^{n-1}.
\end{equation}
 \begin{equation}\label{r02}
\widehat{\Upsilon}_{n}(x)=\log\frac{\mathcal{R}([\xi_{0}, x];\,\, f^{q_{n-1}})}{\mathcal{R}([x, \xi_{q_{n}}]; \,\, f^{q_{n-1}})}
 +\log \widehat{m}_{n},\,\,\,\,\,\,\,\,\,\, x\in I_{0}^{n}.
\end{equation}
To be easy to write we use  the following notations
 $$
 \alpha_{i}:=\xi_{i+q_{n-1}},\,\,\,\,\beta_{i}:=\xi_{i}\,\,\,\,\,
 \text{and}\,\,\,\,\,\, x_{i}:=f^{i}(x)\in I^{n-1}_{i},
\,\,\,\,\,\, \,\,\,0\leq i \leq q_{n}.
 $$
 $$
 \bar{\alpha}_{j}:=\xi_{j},\,\,\,\,\bar{\beta}_{j}:=\xi_{j+q_{n}}\,\,\,\,\,
 \text{and}\,\,\,\,\,\, y_{j}:=f^{j}(y)\in I^{n}_{j},
\,\,\,\,\,\, \,\,\,0\leq j \leq q_{n-1}.
 $$
Hence
 \begin{equation}\label{r01aaaaaa}
 x=\beta_{0}+z_{0}(\alpha_{0}-\beta_{0})\,\,\,\,\,\,\,\,\,\,\ \text{and}
 \,\,\,\,\,\,\,\,\,\,\  y=\bar{\beta}_{0}+\bar{z}_{0}(\bar{\alpha}_{0}-\bar{\beta}_{0}).
 \end{equation}
 So, we set
\begin{equation}\label{r01a}
 \widetilde{\Upsilon}_{n}(z_{0}):=\widetilde{\Upsilon}_{n}(\beta_{0}+z_{0}(\alpha_{0}-\beta_{0})),\,\,\,\,\,\,
\widehat{\Upsilon}_{n}(\bar{z}_{0}):= \widehat{\Upsilon}_{n}(\bar{\beta}_{0}+\bar{z}_{0}(\bar{\alpha}_{0}-\bar{\beta}_{0})).
\end{equation}
Below we estimate $\widetilde{\Upsilon}_{n}$ and $\widehat{\Upsilon}_{n}.$
These estimates will be used in the next section to approximate
relative coordinates with M\"{o}bius transformations.
\begin{lemm}\label{nolinchihosila}
Let $f\in \mathrm{D}^{1+\mathcal{Z}_{\gamma}}(\mathbb{S}^{1}\setminus\{\xi_{0}\})$ and $\gamma\in(0, +\infty).$
Then there exists a constant $C>0$ such that
$$
\underset{z_{0}\in [0,1]}{\max}|\widetilde{\Upsilon}_{n}(z_{0})|\leq  \frac{C}{n^{\gamma}},\,\,\,\,\,\,
\underset{\bar{z}_{0}\in [0,1]}{\max}\,\,|\widehat{\Upsilon}_{n}(\bar{z}_{0})|\leq  \frac{C}{n^{\gamma}}
$$
for all $n\geq 1.$
\end{lemm}
\begin{proof}
We prove only first inequality, the second inequality
can be proved analogously.
Since, the ratio distortion is multiplicative with respect to
composition and
using notation (\ref{ok5}) we have
 \begin{equation}\label{ok8}
\widetilde{\Upsilon}_{n}(z_{0})=\sum_{i=0}^{q_{n}-1}\log\frac{\mathcal{R}_{\alpha_{i}}(x_{i})}
{\mathcal{R}_{\beta_{i}}(x_{i})} +\log \widetilde{m}_{n}.
\end{equation}
Because the sytem of intervals $\{I_{i}^{n-1}=[\alpha_{i}, \beta_{i}], \,\,0\leq i< q_{n}\}$ do not contain the break point,
utilizing   Lemma \ref{raiovahosila} we get
\begin{equation}\label{ok9}
\sum_{i=0}^{q_{n}-1}\log\frac{\mathcal{R}_{\alpha_{i}}(x_{i})}
{\mathcal{R}_{\beta_{i}}(x_{i})}=
\sum_{i=0}^{q_{n}-1}\log\Big(1+\frac{f'(\alpha_{i})-f'(\beta_{i})}{2f'(\beta_{i})}\Big)+
\end{equation}
$$
\mathcal{O}\Big(\sum_{i=0}^{q_{n}-1}|I_{i}^{n-1}|\cdot\mathcal{Z}_{\gamma}(|I_{i}^{n-1}|)+|f'(\alpha_{i})-f'(\beta_{i})|\cdot
\Omega(|I_{i}^{n-1}|, \gamma) \Big).
$$
It is clear
\begin{equation}\label{ok10}
\sum_{i=0}^{q_{n}-1}|I_{i}^{n-1}|\mathcal{Z}_{\gamma}(|I_{i}^{n-1}|)=
\mathcal{O}(\mathcal{Z}_{\gamma}(d_{n-1}))
\end{equation}
and
 \begin{equation}\label{ok11}
\sum_{i=0}^{q_{n}-1}|f'(\alpha_{i})-f'(\beta_{i})|\cdot
\Omega(|I_{i}^{n-1}|, \gamma)=\mathcal{O}\Big(\Omega(d_{n-1}, \gamma)\Big)
\end{equation}
where $d_{n}:=\|f^{q_{n}}-\mathrm{Id}\|_{C^0}.$
It is well known that  $d_{n}=\mathcal{O}(\lambda^{n})$ (see \cite{DSA2013})
and this implies
$$
\mathcal{Z}_{\gamma}(d_{n-1})+\Omega(d_{n-1}, \gamma)=\mathcal{O}\Big(\frac{1}{n^{\gamma}}\Big).
$$
Therefore, using  Taylor's formula and combining (\ref{ok9})-(\ref{ok11}) we get
\begin{equation*}\label{ok14}
\sum_{i=0}^{q_{n}-1}\log\frac{\mathcal{R}_{\alpha_{i}}(x_{i})}
{\mathcal{R}_{\beta_{i}}(x_{i})}=
-\log \widetilde{m}_{n}+\mathcal{O}\Big(\frac{1}{n^{\gamma}}\Big).
\end{equation*}
Thus and so
\begin{equation*}\label{ok15}
\underset{z_{0}\in [0,1]}{\max}|\widetilde{\Upsilon}_{n}(z_{0})|\leq \frac{C}{n^{\gamma}}.
 \end{equation*}
Lemma \ref{nolinchihosila} is proved.
\end{proof}
The following estimates will be used in the next section to approximate
relative coordinates with M\"{o}bius transformations in $C^{1}$-topology.
 \begin{lemm}\label{birinchihosila}
Let $f\in \mathrm{D}^{1+\mathcal{Z}_{\gamma}}(\mathbb{S}^{1}\setminus\{\xi_{0}\})$ and $\gamma\in(0, +\infty)$.
Then there exists a constant $C>0$ and a natural number $N_{0}=N_{0}(f)$ such that
$$
\underset{z_{0}\in [0,1]}{\max}\Big|z_{0}(1-z_{0})\frac{\mathrm{d}\widetilde{\Upsilon}_{n}(z_{0})}{\mathrm{d}z_{0}}\Big|
\leq  \frac{C}{n^{\gamma}},\,\,\,\,\,\,
\underset{\bar{z}_{0}\in [0,1]}{\max}\,\,\Big|\bar{z}_{0}(1-\bar{z}_{0})
\frac{\mathrm{d}\widehat{\Upsilon}_{n}(\bar{z}_{0})}{\mathrm{d}\bar{z}_{0}}\Big|\leq  \frac{C}{n^{\gamma}}
$$
for all $n\geq N_{0}.$
\end{lemm}
\begin{proof}
We prove the first inequality, the second one can be proved analogously.
Setting
\begin{equation}\label{ok17}
\Psi(x_{i}):=\log \frac{\mathcal{R}_{\alpha_{i}}(x_{i})}{\mathcal{R}_{\beta_{i}}(x_{i})}
\end{equation}
rewrite $\widetilde{\Upsilon}_{n}$ as  follows
\begin{equation}\label{ok19}
  \widetilde{\Upsilon}_{n}(z_{0})=\sum_{i=0}^{q_{n}-1}\Psi(x_{i})+\log \widetilde{m}_{n}.
\end{equation}
It is clear
\begin{equation}\label{ok20}
\frac{\mathrm{d}\widetilde{\Upsilon}_{n}(z_{0})}{\mathrm{d}z_{0}}
=\frac{\mathrm{d}\widetilde{\Upsilon}_{n}(x)}{\mathrm{d}x}\cdot\frac{\mathrm{d}x}{\mathrm{d}z_{0}}=
(\alpha_{0}-\beta_{0})\cdot\frac{\mathrm{d}\widetilde{\Upsilon}_{n}(x)}{\mathrm{d}x}
\end{equation}
and
\begin{equation}\label{ok21}
  \frac{\mathrm{d}\Psi(x_{i})}{\mathrm{d}x}=\frac{\mathrm{d}\Psi(x_{i})}{\mathrm{d}x_{i}}\cdot
  \frac{\mathrm{d}x_{i}}{\mathrm{d}x}=
\frac{\mathrm{d}\Psi(x_{i})}{\mathrm{d}x_{i}}\cdot(f^{i}(x))'.
\end{equation}
Using Finzi's inequality (see \cite{ADM2012} for finitely many breaks) we obtain
\begin{equation}\label{ok23}
  e^{-\nu}\leq \frac{(f^{i}(x))'(\alpha_{0}-\beta_{0})}{(\alpha_{i}-\beta_{i})}\leq e^{\nu}\,\,\,\,\,
\text{and}\,\,\,\,\,\, e^{-2\nu}\leq \frac{z_{0}(1-z_{0})}{z_{i}(1-z_{i})}\leq e^{2\nu}.
\end{equation}
From (\ref{ok19})-(\ref{ok23}) we get
\begin{equation}\label{ok24}
\Big|z_{0}(1-z_{0})\frac{\mathrm{d}\widetilde{\Upsilon}_{n}(z_{0})}{\mathrm{d}z_{0}}\Big|\leq
e^{3\nu}\Big|\sum_{i=0}^{q_{n}-1}z_{i}(1-z_{i})(\alpha_{i}-\beta_{i})\frac{\mathrm{d}\Psi(x_{i})}{\mathrm{d}x_{i}}\Big|.
\end{equation}
By differentiating (\ref{ok17}) we have
\begin{equation}\label{ok25}
  \frac{\mathrm{d}\Psi(x_{i})}{\mathrm{d}x_{i}}=
  \frac{1}{\mathcal{R}_{\alpha_{i}}(x_{i})}\cdot\frac{\mathrm{d}\mathcal{R}_{\alpha_{i}}(x_{i})}{\mathrm{d}x_{i}}-
   \frac{1}{\mathcal{R}_{\beta_{i}}(x_{i})}\cdot\frac{\mathrm{d}\mathcal{R}_{\beta_{i}}(x_{i})}{\mathrm{d}x_{i}}=
  \end{equation}
$$
\Big(\frac{1}{\mathcal{R}_{\alpha_{i}}(x_{i})}-\frac{1}{\mathcal{R}_{\beta_{i}}(x_{i})}\Big)
\cdot\frac{\mathrm{d}\mathcal{R}_{\alpha_{i}}(x_{i})}{\mathrm{d}x_{i}}+
\frac{1}{\mathcal{R}_{\beta_{i}}(x_{i})}\cdot
\Big(\frac{\mathrm{d}\mathcal{R}_{\alpha_{i}}(x_{i})}{\mathrm{d}x_{i}}-\frac{\mathrm{d}\mathcal{R}_{\beta_{i}}(x_{i})}{\mathrm{d}x_{i}}\Big).
$$
Utilizing mean value theorem and Theorems \ref{modofcon}, \ref{WZ} we get
\begin{equation}\label{ok26}
  \Big|\frac{1}{\mathcal{R}_{\alpha_{i}}(x_{i})}-\frac{1}{\mathcal{R}_{\beta_{i}}(x_{i})}\Big|=
  \Big|\frac{1}{f'(\check{\alpha}_{i})}-\frac{1}{f'(\check{\beta}_{i})}\Big|\leq C\cdot\Omega(d_{n-1}, \gamma)
\end{equation}
for any $\gamma\in(0, +\infty),$ where $\check{\alpha}_{i}\in [\alpha_{i}, x_{i}]$
and $\check{\beta}_{i}\in [x_{i}, \beta_{i}].$
Using this, we estimate $|\frac{\mathrm{d}\Psi(x_{i})}{\mathrm{d}x_{i}}|$ as
 \begin{equation}\label{ok27}
  \Big|\frac{\mathrm{d}\Psi(x_{i})}{\mathrm{d}x_{i}}\Big|\leq
C\cdot \Omega(d_{n-1}, \gamma)\Big|\frac{\mathrm{d}\mathcal{R}_{\alpha_{i}}(x_{i})}{\mathrm{d}x_{i}}\Big|+
\frac{1}{\underset{\mathbb{S}^1}{\inf} f'(\xi)}\cdot\Big|\frac{\mathrm{d}\mathcal{R}_{\alpha_{i}}(x_{i})}{\mathrm{d}x_{i}}-
\frac{\mathrm{d}\mathcal{R}_{\beta_{i}}(x_{i})}{\mathrm{d}x_{i}}\Big|.
\end{equation}
Applying this inequality to  the right hand side of (\ref{ok24}) we get
\begin{equation}\label{ok28}
\Big|z_{0}(1-z_{0})\frac{\mathrm{d}\widetilde{\Upsilon}_{n}(z_{0})}{\mathrm{d}z_{0}}\Big|\leq
Ce^{3\nu}\Omega(d_{n-1}, \gamma) \sum_{i=0}^{q_{n}-1}z_{i}(1-z_{i})|\alpha_{i}-\beta_{i}|\Big|\frac{\mathrm{d}\mathcal{R}_{\alpha_{i}}(x_{i})}{\mathrm{d}x_{i}}\Big|
\end{equation}
$$
+\frac{e^{3\nu}}{\underset{\mathbb{S}^1}{\inf} f'(\xi)}\sum_{i=0}^{q_{n}-1}z_{i}(1-z_{i})|\alpha_{i}-\beta_{i}|
\Big|\frac{\mathrm{d}\mathcal{R}_{\alpha_{i}}(x_{i})}{\mathrm{d}x_{i}}-
\frac{\mathrm{d}\mathcal{R}_{\beta_{i}}(x_{i})}{\mathrm{d}x_{i}}\Big|:=A_{n}+B_{n}.
$$
Next we estimate $A_{n}$ and $B_{n}.$ To estimate $A_{n}$ we use the following obvious equality
 $$
 z_{i}(1-z_{i})|\alpha_{i}-\beta_{i}|
\Big|\frac{\mathrm{d}\mathcal{R}_{\alpha_{i}}(x_{i})}{\mathrm{d}x_{i}}\Big|=
z_{i}|f'(x_{i})-f'(\check{\alpha}_{i})|.
 $$
Therefore
$$
A_{n}=Ce^{3\nu}\Omega(d_{n-1}, \gamma) \sum_{i=0}^{q_{n}-1}z_{i}|f'(x_{i})-f'(\check{\alpha}_{i})|.
$$
Since $f'$ has bounded  variation and the system of intervals $\{[x_{i}, \check{\alpha}_{i}], \,\,\,\,\,0\leq i <q_{n}\}$
do not intersect, we have
$$
A_{n}=\mathcal{O}\Big(\Omega(d_{n-1}, \gamma)\Big).
$$
Next we estimate $B_{n}.$ By definition of relative coordinate $z_{i}$ we have
$$
z_{i}(1-z_{i})|\alpha_{i}-\beta_{i}|
\Big|\frac{\mathrm{d}\mathcal{R}_{\alpha_{i}}(x_{i})}{\mathrm{d}x_{i}}-
\frac{\mathrm{d}\mathcal{R}_{\beta_{i}}(x_{i})}{\mathrm{d}x_{i}}\Big|=
\frac{(x_{i}-\alpha_{i})(\beta_{i}-x_{i})}{\beta_{i}-\alpha_{i}}
\Big|\frac{\mathrm{d}\mathcal{R}_{\alpha_{i}}(x_{i})}{\mathrm{d}x_{i}}-
\frac{\mathrm{d}\mathcal{R}_{\beta_{i}}(x_{i})}{\mathrm{d}x_{i}}\Big|.
$$
Due to Corollary \ref{corqavariq} we get
$$
\frac{(x_{i}-\alpha_{i})(\beta_{i}-x_{i})}{\beta_{i}-\alpha_{i}}
\Big|\frac{\mathrm{d}\mathcal{R}_{\alpha_{i}}(x_{i})}{\mathrm{d}x_{i}}-
\frac{\mathrm{d}\mathcal{R}_{\beta_{i}}(x_{i})}{\mathrm{d}x_{i}}\Big|=
\mathcal{O}\Big(|I_{i}^{n-1}|\mathcal{Z}_{\gamma}(|I_{i}^{n-1}|)\Big)
$$
for sufficiently lagre $n.$
Hence
$$
B_{n}=\mathcal{O}\Big(\sum_{i=0}^{q_{n}-1}|I_{i}^{n-1}|\mathcal{Z}_{\gamma}(|I_{i}^{n-1}|)\Big)
=\mathcal{O}\Big(\mathcal{Z}_{\gamma}(d_{n-1})\Big)
$$
for sufficiently large $n.$
Finally
$$
A_{n}+B_{n}=\mathcal{O}\Big(\frac{1}{n^{\gamma}}\Big).
$$
Lemma \ref{birinchihosila} is proved.
\end{proof}
In the next two lemmas we estimate the first and second derivatives of
$\widetilde{\Upsilon}_{n}$ and $\widehat{\Upsilon}_{n}$ in the case of $\gamma\in(1, +\infty).$
These estimates will be used in the next section to approximate relative cordinates with
M\"{o}bius transformations in $C^{2}$-topology for $\gamma>1.$
Note that in this case according to Theorem \ref{WZ},
$f'$ is differentiable on $[\xi_{0}, \xi_{0}+1]$ (here $f'(\xi_{0})=f'(\xi_{0}+0)$ and $f'(\xi_{0}+1)=f'(\xi_{0}-0)$)
and the modulus of
continuity of $f''$ is $\mathcal{P}_{\gamma}.$
\begin{lemm}\label{ikkinchihosila}
Let $f\in \mathrm{D}^{1+\mathcal{Z}_{\gamma}}(\mathbb{S}^{1}\setminus\{\xi_{0}\})$ and $\gamma\in(1, +\infty)$.
Then there exists a constant $C>0$ 
such that
$$
\underset{z_{0}\in [0,1]}{\max}\Big|\frac{\mathrm{d}\widetilde{\Upsilon}_{n}(z_{0})}{\mathrm{d}z_{0}}\Big|
\leq  \frac{C}{n^{\gamma-1}},\,\,\,\,\,
\underset{\bar{z}_{0}\in [0,1]}{\max}\,\,\Big|
\frac{\mathrm{d}\widehat{\Upsilon}_{n}(\bar{z}_{0})}{\mathrm{d}\bar{z}_{0}}\Big|\leq \frac{C}{n^{\gamma-1}}
$$
for all $n\geq 1.$
\end{lemm}
\begin{proof}
We prove only the first inequality, the second one can be proved analogously.
The same manner as in proof of Lemma \ref{birinchihosila} we can show that
\begin{equation*}\label{ok31}
\Big|\frac{\mathrm{d}\widetilde{\Upsilon}_{n}(z_{0})}{\mathrm{d}z_{0}}\Big|\leq
e^{\nu}\Big|\sum_{i=0}^{q_{n}-1}(\alpha_{i}-\beta_{i})\frac{\mathrm{d}\Psi(x_{i})}{\mathrm{d}x_{i}}\Big|\leq
Ce^{\nu}\Omega(d_{n-1}, \gamma) \sum_{i=0}^{q_{n}-1}
|\alpha_{i}-\beta_{i}|\Big|\frac{\mathrm{d}\mathcal{R}_{\alpha_{i}}(x_{i})}{\mathrm{d}x_{i}}\Big|
\end{equation*}
$$
+\frac{e^{\nu}}{\underset{\mathbb{S}^{1}}{\inf} f'(\xi)}\sum_{i=0}^{q_{n}-1}|\alpha_{i}-\beta_{i}|
\Big|\frac{\mathrm{d}\mathcal{R}_{\alpha_{i}}(x_{i})}{\mathrm{d}x_{i}}-
\frac{\mathrm{d}\mathcal{R}_{\beta_{i}}(x_{i})}{\mathrm{d}x_{i}}\Big|:=C_{n}+D_{n}.
$$
Below we estimate $C_{n}$ and $D_{n}.$
Since  $f'$ is differentiable,   it can be shown easily
\begin{equation}\label{ok31aa}
|\alpha_{i}-\beta_{i}|\Big|\frac{\mathrm{d}\mathcal{R}_{\alpha_{i}}(x_{i})}{\mathrm{d}x_{i}}\Big|=
|\alpha_{i}-\beta_{i}|\Big|\frac{f'(x_{i})-\mathcal{R}_{\alpha_{i}}(x_{i})}{x_{i}-\alpha_{i}}\Big|=
\end{equation}
$$
|\alpha_{i}-\beta_{i}|\Big|\int_{\alpha_{i}}^{x_{i}}\frac{f''(y)(y-\alpha_{i})}{(x_{i}-\alpha_{i})^{2}}dy\Big|
=\mathcal{O}\Big(|\alpha_{i}-\beta_{i}|\Big).
$$
Thereby
$$
C_{n}=\mathcal{O}\Big(\Omega(d_{n-1}, \gamma)\sum_{i=0}^{q_{n}-1}|\alpha_{i}-\beta_{i}|\Big)
=\mathcal{O}\Big(\Omega(d_{n-1}, \gamma)\Big).
$$
 Now we estimate $D_{n}.$
According to Lemma \ref{ratiohosilaayirmasi} we have
\begin{equation*}\label{ok31aaa}
\Big|\frac{\mathrm{d}\mathcal{R}_{\alpha_{i}}(x_{i})}{\mathrm{d}x_{i}}-
\frac{\mathrm{d}\mathcal{R}_{\beta_{i}}(x_{i})}{\mathrm{d}x_{i}}\Big|\leq
C\mathcal{P}_{\gamma}(|I^{n-1}_{i}|).
\end{equation*}
Hence
$$
D_{n}=\mathcal{O}\Big(\sum_{i=0}^{q_{n}-1}|\alpha_{i}-\beta_{i}|\mathcal{P}_{\gamma}(|I^{n-1}_{i}|)\Big)=
\mathcal{O}\Big(\mathcal{P}_{\gamma}(d_{n-1})\Big).
$$
From the definiteness of $\mathcal{P}_{\gamma}$
$$
\mathcal{P}_{\gamma}(d_{n-1})=\sum_{k=1}^{\infty}\mathcal{Z}_{\gamma}(2^{-k}d_{n-1})\leq
\int_{0}^{1}\frac{\mathcal{Z}_{\gamma}(yd_{n-1})}{y}dy=\frac{1}{\gamma-1}\Big(\log \frac{1}{d_{n-1}}\Big)^{1-\gamma}.
$$
As before, using the relation $d_{n}=\mathcal{O}(\lambda^{n})$ we obtain
$$
\Big(\log \frac{1}{d_{n-1}}\Big)^{1-\gamma}+\Omega(d_{n-1}, \gamma)=\mathcal{O}\Big(\frac{1}{n^{\gamma-1}}\Big).
$$
Finally
$$
C_{n}+D_{n}=\mathcal{O}\Big(\frac{1}{n^{\gamma-1}}\Big).
$$
Lemma \ref{ikkinchihosila} is proved.
\end{proof}
\begin{lemm}\label{ikkinchihosila2}
Let $f\in \mathrm{D}^{1+\mathcal{Z}_{\gamma}}(\mathbb{S}^{1}\setminus\{\xi_{0}\})$ and $\gamma\in(1, +\infty)$.
Then there exists a constant $C>0$ and a natural number $N_{0}=N_{0}(f)$ such that
$$
\underset{z_{0}\in [0,1]}{\max}\Big|z_{0}(1-z_{0})\frac{\mathrm{d}^{2}\widetilde{\Upsilon}_{n}(z_{0})}{\mathrm{d}z^{2}_{0}}\Big|\leq  \frac{C}{n^{\gamma-1}},\,\,\,\,\,
\underset{\bar{z}_{0}\in [0,1]}{\max}\,\,\Big|
\bar{z}_{0}(1-\bar{z}_{0})\frac{\mathrm{d}^{2}\widehat{\Upsilon}_{n}(\bar{z}_{0})}{\mathrm{d}\bar{z}^{2}_{0}}\Big|\leq
 \frac{C}{n^{\gamma-1}}
$$
for all $n\geq N_{0}.$
\end{lemm}
\begin{proof}
As above we prove only the first inequality, the second one can be proved analogously.
By differentiating (\ref{ok20}) and (\ref{ok21}) we obtain
\begin{equation*}\label{ok34}
\frac{\mathrm{d}^{2}\widetilde{\Upsilon}_{n}(z_{0})}{\mathrm{d}z^{2}_{0}}
=(\alpha_{0}-\beta_{0})^{2}\cdot\frac{\mathrm{d}^{2}\widetilde{\Upsilon}_{n}(x)}{\mathrm{d}x^{2}},
\end{equation*}
and
\begin{equation*}\label{ok35}
  \frac{\mathrm{d}^{2}\Psi(x_{i})}{\mathrm{d}x^{2}}=\frac{\mathrm{d}^{2}\Psi(x_{i})}{\mathrm{d}x^{2}_{i}}\cdot
  \Big(\frac{\mathrm{d}x_{i}}{\mathrm{d}x}\Big)^{2}+
  \frac{\mathrm{d}\Psi(x_{i})}{\mathrm{d}x_{i}}\cdot\frac{\mathrm{d}^{2}x_{i}}{\mathrm{d}x^{2}}.
\end{equation*}
Thus
$$
z_{0}(1-z_{0})\frac{\mathrm{d}^{2}\widetilde{\Upsilon}_{n}(z_{0})}{\mathrm{d}z^{2}_{0}}=
z_{0}(1-z_{0})(\alpha_{0}-\beta_{0})^{2}\sum_{i=0}^{q_{n}-1}\frac{\mathrm{d}^{2}\Psi(x_{i})}{\mathrm{d}x^{2}_{i}}\cdot
  \Big(\frac{\mathrm{d}x_{i}}{\mathrm{d}x}\Big)^{2}+
$$
$$
z_{0}(1-z_{0})(\alpha_{0}-\beta_{0})^{2}\sum_{i=0}^{q_{n}-1}
\frac{\mathrm{d}\Psi(x_{i})}{\mathrm{d}x_{i}}\cdot\frac{\mathrm{d}^{2}x_{i}}{\mathrm{d}x^{2}}:=E_{n}+F_{n}.
$$
Next we estimate $E_{n}$ and $F_{n}.$ Using  (\ref{ok23}) we get
\begin{equation}\label{ok36}
z_{0}(1-z_{0})(\alpha_{0}-\beta_{0})^{2}\Big(\frac{\mathrm{d}x_{i}}{\mathrm{d}x}\Big)^{2}
\leq e^{4\nu}z_{i}(1-z_{i})(\alpha_{i}-\beta_{i})^{2}
\end{equation}
for any $0\leq i<q_{n}.$ By differentiating (\ref{ok25}) we obtain
$$
\frac{\mathrm{d}^{2}\Psi(x_{i})}{\mathrm{d}x^{2}_{i}}=
-\frac{1}{\mathcal{R}^{2}_{\alpha_{i}}(x_{i})}\cdot \Big(\frac{\mathrm{d}\mathcal{R}_{\alpha_{i}}(x_{i})}{\mathrm{d}x_{i}}\Big)^{2}+
\frac{1}{\mathcal{R}_{\alpha_{i}}(x_{i})}\cdot \frac{\mathrm{d}^{2}\mathcal{R}_{\alpha_{i}}(x_{i})}{\mathrm{d}x^{2}_{i}}+
$$
\begin{equation}\label{ok37}
\frac{1}{\mathcal{R}^{2}_{\beta_{i}}(x_{i})}\cdot \Big(\frac{\mathrm{d}\mathcal{R}_{\beta_{i}}(x_{i})}{\mathrm{d}x_{i}}\Big)^{2}-
\frac{1}{\mathcal{R}_{\beta_{i}}(x_{i})}\cdot \frac{\mathrm{d}^{2}\mathcal{R}_{\beta_{i}}(x_{i})}{\mathrm{d}x^{2}_{i}}=
\end{equation}
$$
\Big(\frac{1}{\mathcal{R}^{2}_{\beta_{i}}(x_{i})}-
\frac{1}{\mathcal{R}^{2}_{\alpha_{i}}(x_{i})}\Big)
\cdot \Big(\frac{\mathrm{d}\mathcal{R}_{\alpha_{i}}(x_{i})}{\mathrm{d}x_{i}}\Big)^{2}+
\Big(\Big(\frac{\mathrm{d}\mathcal{R}_{\beta_{i}}(x_{i})}{\mathrm{d}x_{i}}\Big)^{2}-
 \Big(\frac{\mathrm{d}\mathcal{R}_{\alpha_{i}}(x_{i})}{\mathrm{d}x_{i}}\Big)^{2}\Big)
\cdot\frac{1}{\mathcal{R}^{2}_{\beta_{i}}(x_{i})}+
$$
$$
\Big(\frac{1}{\mathcal{R}_{\alpha_{i}}(x_{i})}-\frac{1}{\mathcal{R}_{\beta_{i}}(x_{i})}\Big)
\cdot \frac{\mathrm{d}^{2}\mathcal{R}_{\alpha_{i}}(x_{i})}{\mathrm{d}x^{2}_{i}}+
\Big(\frac{\mathrm{d}^{2}\mathcal{R}_{\alpha_{i}}(x_{i})}{\mathrm{d}x^{2}_{i}}- \frac{\mathrm{d}^{2}\mathcal{R}_{\beta_{i}}(x_{i})}{\mathrm{d}x^{2}_{i}}\Big)
\cdot\frac{1}{\mathcal{R}_{\beta_{i}}(x_{i})}:=E^{(1)}_{n}+E^{(2)}_{n}+E^{(3)}_{n}+E^{(4)}_{n}.
$$
We multiply each  $E^{(s)}_{n},$ $s=1,2,3,4$  to the right hand side of (\ref{ok36})
and estimate them separately.
Relations (\ref{ok26}) and (\ref{ok31aa}) imply
\begin{equation}\label{ok38}
e^{4v}\sum_{i=0}^{q_{n}-1}z_{i}(1-z_{i})(\alpha_{i}-\beta_{i})^{2}|E^{(1)}_{n}|\leq
C d_{n-1}\Omega(d_{n-1}, \gamma).
\end{equation}
From the second inequality of Lemma \ref{ratiohosilaayirmasi} it follows
\begin{equation}\label{ok39}
e^{4v}\sum_{i=0}^{q_{n}-1}z_{i}(1-z_{i})(\alpha_{i}-\beta_{i})^{2}|E^{(2)}_{n}|\leq C d_{n-1}\mathcal{P}_{\gamma}(d_{n-1}).
\end{equation}
As we have seen in the proof of Lemma \ref{ratiohosilaayirmasi}, we have
$$
\frac{\mathrm{d}^{2}\mathcal{R}_{\alpha_{i}}(x_{i})}{\mathrm{d}x^{2}_{i}}=
2\int_{\alpha_{i}}^{x_{i}}\frac{(f''(x_{i})-f''(y))(y-\alpha_{i})}{(x_{i}-\alpha_{i})^{3}}dy.
$$
By definition of $(1-z_{i})$ and the second inequality of Lemma \ref{ratiohosilaayirmasi},
we obtain
$$
\Big|(1-z_{i})(\alpha_{i}-\beta_{i})\frac{\mathrm{d}^{2}\mathcal{R}_{\alpha_{i}}(x_{i})}{\mathrm{d}x^{2}_{i}}\Big|
\leq 2 \int_{\alpha_{i}}^{x_{i}}\frac{|f''(x_{i})-f''(y)|(y-\alpha_{i})}{(x_{i}-\alpha_{i})^{2}}dy
\leq C \mathcal{P}_{\gamma}(|I^{n-1}_{i}|).
$$
Therefore, the last two relations and inequality (\ref{ok26}) imply
\begin{equation}\label{ok39i}
e^{4v}\sum_{i=0}^{q_{n}-1}z_{i}(1-z_{i})(\alpha_{i}-\beta_{i})^{2}
|E^{(3)}_{n}|\leq C \mathcal{P}_{\gamma}(d_{n-1})\Omega(d_{n-1}, \gamma).
\end{equation}
From the first  inequality of  Lemma \ref{ratiohosilaayirmasi} it follows
\begin{equation}\label{ok40}
e^{4v}\sum_{i=0}^{q_{n}-1}z_{i}(1-z_{i})(\alpha_{i}-\beta_{i})^{2}|E^{(4)}_{n}|
\leq
\end{equation}
$$
C\sum_{i=0}^{q_{n}-1}(x_{i}-\alpha_{i})(\beta_{i}-x_{i})\Big|\frac{\mathrm{d}^{2}\mathcal{R}_{\alpha_{i}}(x_{i})}{\mathrm{d}x^{2}_{i}}- \frac{\mathrm{d}^{2}\mathcal{R}_{\beta_{i}}(x_{i})}{\mathrm{d}x^{2}_{i}}\Big|
\leq C \mathcal{P}_{\gamma}(d_{n-1}).
$$
Relations (\ref{ok36})-(\ref{ok40}) then imply the following estimate
$$
|E_{n}|=\mathcal{O}\Big(\frac{1}{n^{\gamma-1}}\Big).
$$
We pass to estimate $F_{n}.$ First we show  the validity of the following inequalities
\begin{equation}\label{ok41}
e^{-2\nu}\inf_{x\in \mathbb{S}^{1}}\Big|\frac{f''(x)}{f'(x)}\Big|\cdot|\alpha_{i}-\beta_{i}|\leq
(\alpha_{0}-\beta_{0})^{2}\Big|\frac{\mathrm{d}^{2}x_{i}}{\mathrm{d}x^{2}}\Big|\leq e^{2\nu}\sup_{x\in \mathbb{S}^{1}}\Big|\frac{f''(x)}{f'(x)}\Big|\cdot|\alpha_{i}-\beta_{i}|
\end{equation}
for any $0\leq i< q_{n}.$
Consider the function
$$
\mathcal{H}_{i}(x)=\sum_{j=0}^{i-1}(f^{j}(x))'
$$
where $x\in I_{0}^{n-1}$ and $0\leq i< q_{n}.$
Using Finzi's inequality it can be easily  shown
\begin{equation}\label{ok42}
e^{-\nu}\leq\frac{\mathcal{H}_{i}(x)}{\mathcal{H}_{i}(y)}\leq e^{\nu}
\end{equation}
for any $x,y\in I_{0}^{n-1}$  and $0\leq i< q_{n}.$
On the other hand
$$
\int_{I_{0}^{n-1}}\mathcal{H}_{i}(x)dx=\sum_{j=0}^{i-1}|I_{j}^{n-1}|.
$$
This and  inequality  (\ref{ok42}) imply
\begin{equation}\label{ok43}
e^{-\nu}\frac{1}{|I_{0}^{n-1}|}\sum_{j=0}^{i-1}|I_{j}^{n-1}|\leq \mathcal{H}_{i}(x)\leq e^{\nu}\frac{1}{|I_{0}^{n-1}|}\sum_{j=0}^{i-1}|I_{j}^{n-1}|.
\end{equation}
We find $\frac{\mathrm{d}^{2}x_{i}}{\mathrm{d}x^{2}}:$
$$
\frac{\mathrm{d}^{2}x_{i}}{\mathrm{d}x^{2}}=(f^{i}(x))'\sum_{j=0}^{i-1}\frac{f''(x_{j})}{f'(x_{j})}(f^{j}(x))'.
$$
Inequalities (\ref{ok23}) and (\ref{ok43}) imply
$$
(\alpha_{0}-\beta_{0})^{2}\Big|\frac{\mathrm{d}^{2}x_{i}}{\mathrm{d}x^{2}}\Big|
\leq  \sup_{x\in \mathbb{S}^{1}}\Big|\frac{f''(x)}{f'(x)}\Big|(\alpha_{0}-\beta_{0})^{2} (f^{i}(x))'\,\mathcal{H}_{i}(x)
\leq e^{2\nu}\sup_{x\in \mathbb{S}^{1}} \Big|\frac{f''(x)}{f'(x)}\Big|\cdot|\alpha_{i}-\beta_{i}|.
$$
The left hand side of  (\ref{ok41}) is proved  similarly.
Now we  continue  estimating $F_{n}.$
After above  preparations we have
$$
|F_{n}|\leq C\sum_{i=0}^{q_{n}-1}
z_{i}(1-z_{i})|\alpha_{i}-\beta_{i}|\Big|\frac{\mathrm{d}\Psi(x_{i})}{\mathrm{d}x_{i}}\Big|.
$$
This and  inequalities (\ref{ok25})-(\ref{ok28}) imply
$$
|F_{n}|\leq C(A_{n}+B_{n}).
$$
As we have shown in the proof of Lemma \ref{birinchihosila}
$$
A_{n}+B_{n}=\mathcal{O}\Big(\frac{1}{n^{\gamma}}\Big)
$$
for sufficiently large $n.$ Therefore
$$
|F_{n}|\leq \frac{C}{n^{\gamma}}
$$
for sufficiently large $n.$
Finally we have
$$
|E_{n}|+|F_{n}|=\mathcal{O}(\frac{1}{n^{\gamma-1}})
$$
for sufficiently large $n.$
 Lemma \ref{ikkinchihosila2} is therefore completely proved.
\end{proof}

\section{Comparing relative coordinates with M\"{o}bius transformations}
In this section we show that the relative coordinates $z_{q_{n}}(z_{0})$ and
$\hat{z}_{q_{n-1}}(\bar{z}_{0})$ are approximated by M\"{o}bius transformations. To characterize these
approximations  we define a M\"{o}bius map
$\mathcal{M}_{\mathcal{T}}$ as follows
\begin{equation}\label{ok44}
\mathcal{M}_{\mathcal{T}}(z)=\frac{z \mathcal{T}}{1+z(\mathcal{T}-1)}.
\end{equation}
\begin{lemm}\label{relationmobius}
Let $f\in \mathrm{D}^{1+\mathcal{Z}_{\gamma}}(\mathbb{S}^{1}\setminus\{\xi_{0}\})$ and $\gamma\in(0, 1].$
Then there exists a constant $C>0$ and a natural number $N_{0}=N_{0}(f)$
such that
$$
\|z_{q_{n}}-\mathcal{M}_{\widetilde{m}_n}\|_{C^{1}([0,1])}\leq \frac{C}{n^{\gamma}},\,\,\,\,\,\,\,\,\,\,
\|\hat{z}_{q_{n-1}}- \mathcal{M}_{\widehat{m}_n}\|_{C^{1}([0,1])}\leq \frac{C}{n^{\gamma}}
$$
for all $n\geq N_{0}.$
\end{lemm}
\begin{proof}
We prove the first inequality. For this  we  find the explicit form of $z_{q_{n}}(z_{0}).$
After simple  computations   we get
$$
\frac{1-z_{q_{n}}}{z_{q_{n}}}\cdot\frac{z_{0}}{1-z_{0}}=\frac{\mathcal{R}([\xi_{q_{n-1}}, x];f^{q_{n}})}{\mathcal{R}([x, \xi_{0}];f^{q_{n}})}.
$$
On the other hand, the relation (\ref{r01})  implies
$$
\frac{\mathcal{R}([\xi_{q_{n-1}}, x];f^{q_{n}})}{\mathcal{R}([x, \xi_{0}];f^{q_{n}})}=
\frac{1}{\widetilde{m}_{n}}\exp(\widetilde{\Upsilon}_{n}(z_{0})).
$$
Therefore, the last two relations imply
$$
\frac{1-z_{q_{n}}}{z_{q_{n}}}\cdot\frac{z_{0}}{1-z_{0}}=
\frac{1}{\widetilde{m}_{n}}\exp(\widetilde{\Upsilon}_{n}(z_{0})).
$$
Solving for $z_{q_{n}},$ we get
\begin{equation}\label{ok47}
z_{q_{n}}(z_{0})=\frac{z_{0}\widetilde{m}_{n}}{(1-z_{0})\exp(\widetilde{\Upsilon}_{n}(z_{0}))+z_{0}\widetilde{m}_{n}}.
\end{equation}
Using Lemma \ref{nolinchihosila} we get
\begin{equation}\label{ok48}
\max_{z_{0}\in [0,1]}\Big|z_{q_{n}}(z_{0})-\mathcal{M}_{\widetilde{m}_n}(z_{0})\Big|\leq \frac{C}{n^{\gamma}}.
\end{equation}
for all $n\geq 1.$
By differentiating (\ref{ok47}) we obtain
$$
z'_{q_{n}}(z_{0})=\frac{\Big(1-z_{0}(1-z_{0})\widetilde{\Upsilon}'_{n}(z_{0})\Big)\widetilde{m}_{n}
\exp(\widetilde{\Upsilon}_{n}(z_{0}))}{\Big((1-z_{0})\exp(\widetilde{\Upsilon}_{n}(z_{0}))+z_{0}\widetilde{m}_{n}\Big)^{2}}.
$$
Utilizing  Lemma \ref{birinchihosila} 
\begin{equation}\label{ok50}
\max_{z_{0}\in [0,1]}\Big|z'_{q_{n}}(z_{0})-\mathcal{M}'_{\widetilde{m}_n}(z_{0})\Big|\leq \frac{C}{n^{\gamma}}
\end{equation}
for all $n\geq N_{0}.$
Inequalities (\ref{ok48}) and (\ref{ok50}) imply the proof of first inequality of Lemma \ref{relationmobius}.
The proof of the second inequality is similar.
\end{proof}
Now we consider the case $\gamma>1.$
\begin{lemm}\label{relationmobius2}
Let $f\in \mathrm{D}^{1+\mathcal{Z}_{\gamma}}(\mathbb{S}^{1}\setminus\{\xi_{0}\})$ and  $\gamma>1.$
Then there exists a constant $C>0$ and natural number $N_{0}=N_{0}(f)$ such that for all $n\geq N_{0}$
the following inequalities hold
\begin{equation}\label{ok51}
\|z_{q_{n}}-\mathcal{M}_{m_n}\|_{C^{1}([0,1])}\leq  \frac{C}{n^{\gamma}},\,\,\,\,\,\,\,\,\,\,
\|\hat{z}_{q_{n-1}}- \mathcal{M}_{\frac{c_{n}}{m_n}}\|_{C^{1}([0,1])}\leq \frac{C}{n^{\gamma}}.
\end{equation}
Moreover,
 \begin{equation}\label{ok56}
\|z''_{q_{n}}-\mathcal{M}''_{m_n}\|_{C^{0}([0,1])}\leq  \frac{C}{n^{\gamma-1}},\,\,\,\,\,\,\,\,\,\,
\|\hat{z}''_{q_{n-1}}- \mathcal{M}''_{\frac{c_{n}}{m_n}}\|_{C^{0}([0,1])}\leq  \frac{C}{n^{\gamma-1}}.
\end{equation}
\end{lemm}
\begin{proof}
To prove this lemma we use the explicit forms of $z_{q_{n}}(z_{0}),$ $z'_{q_{n}}(z_{0})$
and $\hat{z}_{q_{n-1}}(\bar{z}_{0}),$ $\hat{z}'_{q_{n-1}}(\bar{z}_{0}).$
As we have shown above
\begin{equation}\label{ok52}
z_{q_{n}}(z_{0})=\frac{z_{0}\widetilde{m}_{n}}{(1-z_{0})\exp(\widetilde{\Upsilon}_{n}(z_{0}))+z_{0}\widetilde{m}_{n}},
\end{equation}
\begin{equation}\label{ok49}
z'_{q_{n}}(z_{0})=\frac{\Big(1-z_{0}(1-z_{0})\widetilde{\Upsilon}'_{n}(z_{0})\Big)\widetilde{m}_{n}
\exp(\widetilde{\Upsilon}_{n}(z_{0}))}{\Big((1-z_{0})\exp(\widetilde{\Upsilon}_{n}(z_{0}))+z_{0}\widetilde{m}_{n}\Big)^{2}}.
\end{equation}
The same manner as in the proof of above lemma, it can be found
\begin{equation}\label{ok53}
\hat{z}_{q_{n-1}}(\bar{z}_{0})=\frac{\bar{z}_{0}\widehat{m}_{n}}{(1-\bar{z}_{0})\exp(\widehat{\Upsilon}_{n}(\bar{z}_{0}))+\bar{z}_{0}\widehat{m}_{n}}.
\end{equation}
By differentiating this we obtain
\begin{equation}\label{ok49a}
\hat{z}'_{q_{n-1}}(\bar{z}_{0})=\frac{\Big(1-\bar{z}_{0}(1-\bar{z}_{0})\widehat{\Upsilon}'_{n}(\bar{z}_{0})\Big)\widehat{m}_{n}
\exp(\widehat{\Upsilon}_{n}(\bar{z}_{0}))}{\Big((1-\bar{z}_{0})\exp(\widehat{\Upsilon}_{n}(\bar{z}_{0}))+\bar{z}_{0}\widehat{m}_{n}\Big)^{2}}.
\end{equation}
Now we compare $\widetilde{m}_{n},$ $\widehat{m}_{n}$ with $m_{n},$  $c_{n}m^{-1}_{n}$ respectively.
By assumption of lemma $\gamma>1,$ therefore according to Theorem \ref{WZ},
$f'$ is differentiable. Thus
$$
\Big|\log \widetilde{m}_{n}- \log m_{n}\Big|\leq
\frac{1}{2}\sum_{i=0}^{q_{n}-1}\int_{I^{n-1}_{i}}\Big|\frac{f''(x)}{f'(\xi_{i})} -\frac{f''(x)}{f'(x)}\Big|dx.
$$
Theorem \ref{modofcon} implies
 $$
 \int_{I^{n-1}_{i}}\Big|\frac{f''(x)}{f'(\xi_{i})} -\frac{f''(x)}{f'(x)}\Big|dx\leq
 \frac{\Omega(d_{n-1}, \gamma)}{(\underset{\mathbb{S}^{1}}{\inf} f'(x))^{2}}
 \int_{I^{n-1}_{i}}|f''(x)|dx.
 $$
 It is clear that $\Omega(d_{n-1}, \gamma)=\mathcal{O}(\lambda^{n})$ for $\gamma>1.$
 Hence
 \begin{equation}\label{ok54}
 \widetilde{m}_{n}=m_{n}+\mathcal{O}(\lambda^{n}).
 \end{equation}
 Similarly  it can be shown
 $$
 \widehat{m}_{n}=\exp\Big((-1)^{n}\sum_{j=0}^{q_{n-1}-1}\int_{I^{n}_{j}}\frac{f''(x)}{2f'(x)}dx\Big)+\mathcal{O}(\lambda^{n}).
 $$
   On the other hand
 $$
 m_{n}\cdot\exp\Big((-1)^{n}\sum_{j=0}^{q_{n-1}-1}\int_{I^{n}_{j}}\frac{f''(x)}{2f'(x)}dx\Big)=
 \exp\Big((-1)^{n}\int_{\mathbb{S}^{1}}\frac{f''(x)}{2f'(x)}dx\Big)=c_{n}.
 $$
Therefore
\begin{equation}\label{ok55}
 \widetilde{m}_{n}=\frac{c_{n}}{m_{n}}+\mathcal{O}(\lambda^{n}).
 \end{equation}
Finally, relations (\ref{ok52})-(\ref{ok55}) together with Lemmas \ref{nolinchihosila} and  \ref{birinchihosila}
imply the first assertion of  Lemma \ref{relationmobius2}.
To prove the second assertion,  we find the explicit forms of $z''_{q_{n}}$ and $\hat{z}''_{q_{n-1}}$
as follows
$$
z''_{q_{n}}(z_{0})=\frac{\widetilde{m}_{n}\exp(\widetilde{\Upsilon}_{n}(z_{0}))\Big(\widetilde{\Upsilon}'_{n}(z_{0})
\Big(2z_{0}-z_{0}(1-z_{0})\widetilde{\Upsilon}'_{n}(z_{0})\Big)
-z_{0}(1-z_{0})\widetilde{\Upsilon}''_{n}(z_{0})\Big)}
{\Big((1-z_{0})\exp(\widetilde{\Upsilon}_{n}(z_{0}))+z_{0}\widetilde{m}_{n}\Big)^{2}}-
$$
$$
\frac{2\widetilde{m}_{n}\exp(\widetilde{\Upsilon}_{n}(z_{0}))\Big(1-z_{0}(1-z_{0})
\widetilde{\Upsilon}'_{n}(z_{0})\Big)\Big(\widetilde{m}_{n}-\exp(\widetilde{\Upsilon}_{n}(z_{0}))+(1-z_{0})\widetilde{\Upsilon}'_{n}(z_{0}) \Big)}{\Big((1-z_{0})\exp(\widetilde{\Upsilon}_{n}(z_{0}))+z_{0}\widetilde{m}_{n}\Big)^{3}}.
$$
Similarly
$$
\hat{z}''_{q_{n}}(\bar{z}_{0})=\frac{\widehat{m}_{n}\exp(\widehat{\Upsilon}_{n}(\bar{z}_{0}))\Big(\widehat{\Upsilon}'_{n}(\bar{z}_{0})
\Big(2\bar{z}_{0}-\bar{z}_{0}(1-\bar{z}_{0})\widehat{\Upsilon}'_{n}(\bar{z}_{0})\Big)
-\bar{z}_{0}(1-\bar{z}_{0})\widehat{\Upsilon}''_{n}(\bar{z}_{0})\Big)}
{\Big((1-\bar{z}_{0})\exp(\widehat{\Upsilon}_{n}(\bar{z}_{0}))+\bar{z}_{0}\widehat{m}_{n}\Big)^{2}}-
$$
$$
\frac{2\widehat{m}_{n}\exp(\widehat{\Upsilon}_{n}(\bar{z}_{0}))\Big(1-\bar{z}_{0}(1-\bar{z}_{0})
\widehat{\Upsilon}'_{n}(\bar{z}_{0})\Big)\Big(\widehat{m}_{n}-\exp(\widehat{\Upsilon}_{n}(\bar{z}_{0}))+(1-\bar{z}_{0})
\widehat{\Upsilon}'_{n}(\bar{z}_{0}) \Big)}{\Big((1-\bar{z}_{0})\exp(\widehat{\Upsilon}_{n}(\bar{z}_{0}))+\bar{z}_{0}
\widehat{m}_{n}\Big)^{3}}.
$$
Using relations (\ref{ok54}), (\ref{ok55})
together with Lemmas \ref{nolinchihosila}, \ref{birinchihosila}, \ref{ikkinchihosila} and
\ref{ikkinchihosila2} we obtain inequalities (\ref{ok56}).
Lemma \ref{relationmobius2} is therefore completely proved.
 \end{proof}
\section{Proofs of main theorems}
In this section we provide the proofs of our main theorems.
Note that the proofs follow closely  that of \cite{KV1991}.
The flows of proofs are as follow: first we introduce a new renormalized coordinate $z$ on
the  $\mathcal{V}_{n}=I_{0}^{n}\cup I_{0}^{n-1}$
and then express the functions $f_{n}$ and $g_{n}$
in terms of renormalized coordinate.
Finally, using relations between new renormalized coordinate $z$ and relative
coordinates $z_{0},$ $\bar{z}_{0}$ and applying  Lemmas \ref{relationmobius},
\ref{relationmobius2}
we obtain the proof of  Theorem \ref{main1}.
To prove Theorem \ref{coefbog} we utilize  Theorems \ref{modofcon} and \ref{WZ}.

\textbf{Proof of  Theorem \ref{main1}.}
Renormalized coordinate $z$ on $\mathcal{V}_{n}$
is defined by
$$
x=\xi_{0}+z(\xi_{0}-\xi_{q_{n-1}}).
$$
In this new coordinate the points $\xi_{q_{n-1}}$ and $\xi_{0}$ go to
$(-1)$ and $0$ respectively.  Denote by $a_{n}$ and $(-b_{n})$ the new
coordinates of  points $\xi_{q_{n}}$ and $\xi_{q_{n}+q_{n-1}}$ i.e.,
$$
a_{n}=\frac{\xi_{q_{n}}-\xi_{0}}{\xi_{0}-\xi_{q_{n-1}}},\,\,\,\,\,\,\,\,\,\,\,
-b_{n}=\frac{\xi_{q_{n}+q_{n-1}}-\xi_{0}}{\xi_{0}-\xi_{q_{n-1}}}.
$$
It is clear that for any $x\in [\xi_{q_{n-1}},\xi_{0}]$ there exists a unique  relative  coordinate $z_{0}\in [0, 1]$
and  renormalized coordinate $z\in [-1, 0]$ which  correspond to $x.$
Similarly, for any $y\in [\xi_{0}, \xi_{q_{n}}]$ there exists  $\bar{z}_{0}\in [0, 1]$
and  $z\in [0, a_{n}].$
 Using definitions of relative  coordinate and renormalized coordinate
 one can show that
$$
 z_{0}=-z, \,\,\,\,\,\,\,\bar{z}_{0}=1-\frac{z}{a_{n}}.
$$
 By definiteness of  $f_{n}$ and the explicit form of $z_{q_{n}},$ we have
 \begin{equation}\label{7.1}
 f_{n}(z)=\frac{f^{q_{n}}(x)-\xi_{0}}{\xi_{0}-\xi_{q_{n-1}}}=
 \frac{\xi_{q_{n}}-\xi_{0}}{\xi_{0}-\xi_{q_{n-1}}}-\frac{\xi_{q_{n}}-f^{q_{n}}(x)}{\xi_{q_{n}}-\xi_{q_{n}+q_{n-1}}}
 \cdot\frac{\xi_{q_{n}}-\xi_{q_{n}+q_{n-1}}}{\xi_{0}-\xi_{q_{n-1}}}=
 \end{equation}
 $$
 a_{n}-(a_{n}+b_{n})z_{q_{n}}(z_{0})=a_{n}-(a_{n}+b_{n})z_{q_{n}}(-z).
 $$
 Analogously it can be shown, that
 \begin{equation}\label{7.2}
 g_{n}(z)=-b_{n}-(1-b_{n})\hat{z}_{q_{n-1}}(1-\frac{z}{a_{n}}).
\end{equation}
 On the other hand an easy computation shows that the
 functions $\widetilde{F}_{n},$ $\widetilde{C}_{n}$ can be represented
 by M\"{o}bius transformations $\mathcal{M}_{\widetilde{m}_{n}},$ $\mathcal{M}_{\widehat{m}_{n}}$
 respectively as follows
\begin{equation}\label{7.3}
\widetilde{F}_{n}(z)=a_{n}-(a_{n}+b_{n})\mathcal{M}_{\widetilde{m}_n}(-z),\,\,\,\,\,\,\,\,
\widetilde{C}_{n}(z)=-b_{n}-(1-b_{n})\mathcal{M}_{\widehat{m}_n}(1-\frac{z}{a_{n}}).
\end{equation}
Relations (\ref{7.1})-(\ref{7.3}) imply
\begin{equation}\label{7.4}
f_{n}(z)-\widetilde{F}_{n}(z)=-(a_{n}+b_{n})\Big(z_{q_{n}}(-z)-\mathcal{M}_{\widetilde{m}_n}(-z)\Big),
\,\,\,\,\,\,\,\,\,\,z\in[-1,0],
\end{equation}
\begin{equation}\label{7.5}
g_{n}(z)-\widetilde{G}_{n}(z)=-(1-b_{n})\Big(\hat{z}_{q_{n-1}}(1-\frac{z}{a_{n}})-\mathcal{M}_{\widehat{m}_n}(1-\frac{z}{a_{n}})\Big),
\,\,\,\,\,\,\,\,\,\,z\in[0, a_{n}].
\end{equation}
By differentiating these  we obtain
\begin{equation}\label{7.6}
f'_{n}(z)-\widetilde{F}'_{n}(z)=(a_{n}+b_{n})\Big(z'_{q_{n}}(-z)-\mathcal{M}'_{\widetilde{m}_n}(-z)\Big),
\,\,\,\,\,\,\,\,\,\,z\in[-1,0],
\end{equation}
\begin{equation}\label{7.7}
g'_{n}(z)-\widetilde{G}'_{n}(z)=\frac{1-b_{n}}{a_{n}}\Big(\hat{z}'_{q_{n-1}}(1-\frac{z}{a_{n}})-\mathcal{M}'_{\widehat{m}_n}(1-\frac{z}{a_{n}})\Big),
\,\,\,\,\,\,\,\,\,\,z\in[0, a_{n}].
\end{equation}
Using Denjoy's inequality (see \cite{ADM2012}) and properties of dynamical partition one can obtain
\begin{equation}\label{7.8}
a_{n}+b_{n}\leq e^{\nu},\,\,\,\,\,\frac{1-b_{n}}{a_{n}}\leq e^{\nu}\,\,\,\,\, \mathrm{and}\,\,\,\,\, 0<1-b_{n}<1.
\end{equation}
The last relations together with (\ref{7.4})-(\ref{7.7}) and  Lemma \ref{relationmobius} imply
$$
\|f_{n}-\widetilde{F}_{n}\|_{C^{1}([-1,0])}\leq \frac{C}{n^{\gamma}},\,\,\,\,\,\,
\|g_{n}-\widehat{G}_{n}\|_{C^{1}([0, a_{n}])}\leq \frac{C}{n^{\gamma}}
$$
for all $n\geq N_{0}.$ Theorem \ref{main1} is  proved.

\textbf{Proof of  Theorem \ref{coefbog}.}
From Lemma \ref{nolinchihosila}
and the relations (\ref{ok54}), (\ref{ok55}), (\ref{7.4})-(\ref{7.8})
directly follow the inequalities (\ref{Ren100}) of Theorem \ref{coefbog}.
To prove the inequalities  (\ref{Ren2})
 we use Theorem \ref{WZ}. According to that theorem, $f'$ is differentiable, hence
 $z'_{q_{n}}$ and $\hat{z}'_{q_{n-1}}$  are differentiable.
By differentiating (\ref{7.6}), (\ref{7.7})   we obtain
\begin{equation*}\label{8.5}
f''_{n}(z)-\widetilde{F}''_{n}(z)=-(a_{n}+b_{n})\Big(z''_{q_{n}}(-z)-\mathcal{M}''_{\widetilde{m}_n}(-z)\Big),
\,\,\,\,\,\,\,\,\,\,z\in[-1,0]
\end{equation*}
\begin{equation*}\label{8.6}
g''_{n}(z)-\widetilde{G}''_{n}(z)=-\frac{1-b_{n}}{a^{2}_{n}}\Big(\hat{z}''_{q_{n-1}}(1-\frac{z}{a_{n}})-\mathcal{M}''_{\widehat{m}_n}(1-\frac{z}{a_{n}})\Big),
\,\,\,\,\,\,\,\,\,\,z\in[0, a_{n}].
\end{equation*}
These together with relations (\ref{ok54}), (\ref{ok55}), (\ref{7.8}) and
the second assertion of Lemma \ref{relationmobius2}
imply
$$
\|f''_{n}-F''_{n}\|_{C^{0}([-1,0])}\leq \frac{C}{n^{\gamma-1}},\,\,\,\,\,\,\,
\|g''_{n}-G''_{n}\|_{C^{0}([0, a_{n}])}\leq \frac{C}{a_{n}n^{\gamma-1}}
$$
for all $n\geq N_{0}.$ The inequalities (\ref{Ren2}) of Theorem \ref{coefbog} are proved.

Now we prove the inequality (\ref{ren3}).
 To prove this we use from the relation between  $f_{n}$ and $g_{n+1}.$
Since these functions correspond to the mapping $f^{q_{n}}$  in  different
coordinate systems, we have
  \begin{equation*}\label{7.9}
  g_{n+1}(z)=-\frac{1}{a_{n}}f_{n}(-a_{n}z).
  \end{equation*}
  This relation implies the following two equalities
  \begin{equation}\label{7.9}
g'_{n+1}(0)=f'_{n}(0),
\end{equation}
\begin{equation}\label{8.0}
g_{n+1}(a_{n+1})=-\frac{1}{a_{n}}f_{n}(-a_{n}a_{n+1}).
\end{equation}
Note that $f'_{n}(0)$ exists and it is equal to $(f^{q_{n}}(\xi_{0}-0))'$ if  $n$ is even and
$(f^{q_{n}}(\xi_{0}+0))'$ if  $n$ is odd.
Using (\ref{ok52})-(\ref{ok49a}) and (\ref{7.1}), (\ref{7.2}) we rewrite (\ref{7.9}), (\ref{8.0}) as
  \begin{equation}\label{8.1}
\frac{(1-b_{n+1})}{a_{n+1}}\cdot\frac{\exp(\widehat{\Upsilon}_{n+1}(1))}{\widehat{m}_{n+1}}=
\frac{(a_{n}+b_{n})\widetilde{m}_{n}}{\exp(\widetilde{\Upsilon}_{n}(0))},
\end{equation}
\begin{equation}\label{8.2}
-b_{n+1}=-1 +\frac{(a_{n}+b_{n})a_{n+1}\widetilde{m}_{n}}
{(1-a_{n}a_{n+1})\exp(\widetilde{\Upsilon}_{n}(a_{n}a_{n+1}))+a_{n}a_{n+1}\widetilde{m}_{n}}.
\end{equation}
From the last equation  we find $(1-b_{n+1})/a_{n+1}$ and then substituting this expression in
(\ref{8.1}) we obtain
$$
\frac{\exp(\widehat{\Upsilon}_{n+1}(1)+\widetilde{\Upsilon}_{n}(0))}{\widehat{m}_{n+1}}
=(1-a_{n}a_{n+1})\exp(\widetilde{\Upsilon}_{n}(a_{n}a_{n+1}))+a_{n}a_{n+1}\widetilde{m}_{n}.
$$
Using this and (\ref{8.2}) we get
\begin{equation*}\label{8.3}
\frac{a_{n+1}}{c_{n+1}}+\frac{b_{n+1}\exp(\widehat{\Upsilon}_{n+1}(1)+\widetilde{\Upsilon}_{n}(0))}{\widehat{m}_{n+1}}-1=
\end{equation*}
$$
a_{n+1}\Big(\frac{1}{c_{n+1}}-a_{n}\exp(\widetilde{\Upsilon}_{n}(a_{n}a_{n+1}))-b_{n}\widetilde{m}_{n}\Big)+
\exp(\widetilde{\Upsilon}_{n}(a_{n}a_{n+1}))-1.
$$
Utilizing Lemma \ref{nolinchihosila} and relations (\ref{ok54}) and (\ref{ok55})
we can show that
\begin{equation}\label{8.4}
a_{n+1}+b_{n+1}m_{n+1}-c_{n+1}=c_{n+1}a_{n+1}(c_{n}-a_{n}-b_{n}m_{n})+\mathcal{O}\Big(\frac{1}{n^{\gamma}}\Big)
\end{equation}
for all $n\geq N_{0}.$
 Set $r_{n+1}:=a_{n+1}+b_{n+1}m_{n+1}-c_{n+1}$ and $a_{n+2}c_{n+2}:=1.$ Iterating relation (\ref{8.4})
we get
$$
r_{n+1}=r_{1}\prod_{i=2}^{n+1}(-a_{i}c_{i})+\mathcal{O}\Big(\sum_{j=3}^{n+2}\prod_{i=j}^{n+2}\frac{(-a_{i}c_{i})}{(i-2)^{\gamma}}\Big).
$$
  It is clear
$$
\Big|\prod_{i=j}^{n+1}(-a_{i}c_{i})\Big|\leq \max\{c, \frac{1}{c}\}a_{n+1}\cdot\frac{|I^{n}_{0}|}{|I^{j-1}_{0}|}.
$$
From well known  fact $\frac{|I^{n}_{0}|}{|I^{j-1}_{0}|}=\mathcal{O}(\lambda ^{n-j+1})$ (see, for instance \cite{ADM2012}) we obtain
$$
|r_{n+1}|=\mathcal{O}\Big(a_{n+1}\sum_{i=1}^{n}\frac{\lambda^{n-i}}{i^{\gamma}}\Big).
$$
On the other hand
$$
\sum_{i=1}^{n}\frac{\lambda^{n-i}}{i^{\gamma}}=\sum_{i=1}^{[\frac{n}{2}]-1}\frac{\lambda^{n-i}}{i^{\gamma}}+
\sum_{i=[\frac{n}{2}]}^{n}\frac{\lambda^{n-i}}{i^{\gamma}}=\mathcal{O}\Big([\frac{n}{2}]\lambda^{[\frac{n}{2}]}
+\frac{1}{[\frac{n}{2}]^{\gamma}} \Big)=\mathcal{O}\Big(\frac{1}{n^{\gamma}}\Big)
$$
where $[\cdot]$ is an integer part of a given number. Hence
$$
|r_{n+1}|=\mathcal{O}\Big(\frac{a_{n+1}}{n^{\gamma}}\Big).
$$
Theorem \ref{coefbog} is completely proved.\\
\\
\textbf{Afterthought.}
At the end of this work, we would like to give our opinion on the further development of our result.
Since  the rate of  convergence is given  in explicit form, we  believe  that this result
 will have  applications in regularity problem of conjugacy.
 Of course, for $\gamma\in (0, 1/2]$ it is  difficult
 to expect the regularity of  conjugacy.
 Because, in this case the second derivative of circle diffeomorphisms can be
 very "bad" (see \cite{Zyg}). However, in the case  $\gamma>1/2$ the situation gets better  that is, in this case
 due to Theorem \ref{WZorginal} (stated in Section 2),  $f'$
 is absolute continuous on $\mathbb{S}^{1}\setminus\{\xi_{0}\}$ and $f''\in L_{p}(\mathbb{S}^{1})$  for every $p>1.$
  Such diffeomorphisms are known as a class of  Katznelson \& Ornstein (KO class)
 in the theory of circle maps. Katznelson \& Ornstein proved that
 diffeomorphisms from KO class are absolute continuously  conjugated with rigid rotation
 for bounded type of irrational rotation numbers \cite{KO1989.2}.
 It is natural  to expect analogues result
 to Katznelson \& Ornstein for conjugacy between two circle diffeomorphisms with breaks.
 In the case $\gamma>1$, in spite of $f$  belongs to $C^{2}(\mathbb{S}^{1}\setminus\{\xi_{0}\})$,
  the convergence rate is sub-exponential. Therefore it is difficult to expect $C^{1}$-rigidity.
 Because, for $C^{1}$-rigidity, it is very important the exponentiality of convergence rate.
 This point  is known from previous  works mentioned in Section 1 and  explained \emph{e.g.} in \cite{MeloFarie I}.
However, it can be expected the  convergence renormalizations with sub-exponential rate.\\
\\
\textbf{Acknowledgements.} The authors would like to
thank professors A. A. Dzhalilov and K. M. Khanin  for making
several useful suggestions. We also wish to express our thanks to the referee
for providing us with helpful comments.
The paper is supported via the grants
 DIP-2014-034 and FRGS/1/2014/ST06/UKM/01/1.

\end{document}